\documentclass[10pt,letterpaper,reqno]{amsart}
\usepackage{color}
\usepackage[latin1]{inputenc}
\usepackage{amsmath}
\usepackage{amsfonts}
\usepackage{amsthm}
\usepackage{epsfig}
\usepackage{graphicx}
\usepackage{amssymb}
\usepackage{bm}
\usepackage{esint}
\usepackage{caption}

\newcommand{\dist}{{\rm dist}}

\newcommand{\spt}{{\rm spt}}

\DeclareMathOperator{\restrict}{\llcorner}
\DeclareMathOperator{\Tan}{Tan}

\DeclareMathOperator{\reach}{reach}

\DeclareMathOperator{\nor}{nor}

\newcommand{\Der}{\mathrm{D}}

\theoremstyle{plain}
\newtheorem{theorem}{Theorem}[section]
\newtheorem{lemma}[theorem]{Lemma}
\newtheorem{corollary}[theorem]{Corollary}

\newtheorem*{theorem*}{Theorem}
\newtheorem*{corollary*}{Corollary}

\theoremstyle{definition}
\newtheorem{definition}{Definition}

\newtheorem{remark}{Remark}[section]

\newtheorem*{notation*}{Notation}

\numberwithin{equation}{section}
\numberwithin{figure}{section}

\usepackage{xcolor}

\title{Finite total curvature and soap bubbles with almost constant higher-order mean curvature}

\author{Mario Santilli}
\address{Dipartimento di Ingegneria e Scienze dell'Informazione e Matematica, Universit\'a degli Studi dell'Aquila, 67100 L'Aquila, Italy}
\email{mario.santilli@univaq.it}

\begin{document}

\begin{abstract}
Given $ n \geq 2 $ and $ k \in \{2, \ldots , n\} $, we study the asymptotic behaviour of sequences of bounded $C^2$-domains of finite total curvature in $ \mathbb{R}^{n+1} $ converging in volume and perimeter, and with the  $ k $-th mean curvature functions converging in $ L^1 $ to a constant. Under natural  mean convexity hypothesis, and assuming an $ L^\infty $-control on the mean curvature outside a set of vanishing area, we prove that finite unions of mutually tangent balls are the only possible limits. This is the first result where such a uniqueness is proved without assuming uniform bounds on the exterior or interior touching balls.
\end{abstract}

\maketitle
\tableofcontents

\section{Introduction}

\subsection{Overview}
If $ \Omega \subseteq \mathbb{R}^{n+1} $ is an open set whose boundary $ \partial \Omega $ is a closed embedded $ C^2 $-hypersurface and if $ \kappa_{\Omega,1} \leq \ldots  \leq \kappa_{\Omega, n} $ are the principal curvatures of $ \partial \Omega $ with respect to the exterior normal of $ \Omega $, then  \emph{the $ k $-th mean curvature function of $ \Omega $} is given by 
$$ H_{\Omega, k} = \sum_{\lambda \in \Lambda(n,k)} \kappa_{\Omega, \lambda(1)}\cdots \kappa_{\Omega, \lambda(k)}, $$
where $ \Lambda(n,k) $ is the set of all increasing maps from $\{1, \ldots , k\} $ to $ \{1, \ldots , n\} $. The function $ H_{\Omega, 1} $ is also called mean curvature of $ \Omega $. 

The following theorem is a classical and well known result.

\medskip 
\noindent {\bf Soap bubble theorem.}	{\it If $ k \in \{1, \ldots, n\} $ and $ \Omega \subseteq \mathbb{R}^{n+1} $ is a bounded and connected  open set with $ C^2 $-boundary such that $ H_{\Omega, k} $ is constant, then $ \Omega $ is a round ball.}

\medskip 

\noindent For $ k = 1 $  this result is due to Alexandrov and was proved using his celebrated moving plane method; see \cite{Aleksandrov}. For arbitrary $ k $ the result was proved in \cite{RosRevista} and \cite{MontielRos}, using an approach based on an optimal geometric inequality inspired by the work of Heintze and Karcher in \cite{HeintzeKarcher}, and on the Minkowski identities \cite{Hsiung}. A proof based on the moving plane method for arbitrary $ k $ is due to Korevaar, see \cite[Appendix]{KorevaarRos}.

Motivated by the soap bubble theorem, we are interested in the following uniqueness problem
\begin{equation*}(UP_k)
	\begin{split}
		&\mbox{{\it Does every sequence $ \Omega_\ell \subseteq \mathbb{R}^{n+1} $ of bounded, connected, open}}
		\\
		&\mbox{{\it smooth sets with bounded perimeters, whose $ k $-th mean curvature}}
		\\
		&\mbox{{\it  functions converge to a constant, have finite unions of}}
		\\
		&\mbox{{\it   mutually tangent balls as their only possible limits?}}
	\end{split}
\end{equation*}
\noindent In general it is not possible to deduce convergence to a single ball from the sole hypothesis of small oscillation of the mean curvature functions. In fact, by truncating and smoothly completing unduloids with thin necks  one can  construct a sequence of bounded and connected smooth boundaries converging to an array of mutually tangent balls while the mean curvatures converge to a constant.

The problem $(UP_1)$ was  thoroughly investigated, even in quantitative ways. In \cite{MaggiDelgadino}, Delgadino and Maggi extended the Alexandrov theorem proving that   if a set of finite perimeter has constant distributional mean curvature, then it is a finite union of closed balls with disjointed interiors.  This result is proved using a measure-theoretic generalization of the Montiel-Ros argument, and implies the following uniqueness result (see \cite[Corollary 2]{MaggiDelgadino}): \emph{finite unions of mutually tangent balls are the only possible limits of sequences of  sets of finite perimeter converging in volume and in  perimeter, and whose distributional mean curvatures converge to a constant.}  Quantitative rates of convergence towards finite unions of balls are obtained in \cite{CiraoloMaggi},  \cite{MaggiArma2018} and in \cite{JulinNinikoski}, employing integral-geometric methods inspired by the Montiel-Ros argument. See also \cite{DeRosaetall} and \cite{MaggiSantilli} for extensions of \cite{MaggiDelgadino} to the anisotropic and Riemannian setting.

The Heintze-Karcher inequality was generalized in the setting of arbitrary closed sets in \cite[Theorem 3.20]{HugSantilli}. This result opens the way to obtain a measure-theoretic version of the soap bubble theorem for sets of positive reach in terms of their curvature measures;  see  \cite[Theorem A, Theorem 6.15]{HugSantilli} (see also Theorem \ref{soap bubble th positive reach} for an extension). As a corollary (see Theorem \ref{positive reach}) we obtain the following answer to $(UP_k)$ for arbitrary $ k $: \emph{one single ball is the only possible limit in the sense of Hausdorff converge, if one assumes that the sets $ \Omega_\ell $ in $(UP_k) $ satisfies a uniform bound on the outer touching balls (i.e.\ lower uniform bound on the reach) and their $ (k-1) $-th mean curvature functions become asymptotically non-negative.} Optimal quantitative rates of convergence in Hausdorff distance towards one single balls can be deduced from the results in \cite{CiraoloVezzoni} (for $ k = 1 $) and  \cite{CiraoloRoncoroniVezzoni} (for $ 1 \leq k \leq n $), assuming a uniform bound both on the interior and exterior touching balls. The results in \cite{CiraoloVezzoni} and \cite{CiraoloRoncoroniVezzoni} are based on a quantitative version of the Alexandrov moving plane method. See also \cite{MagnaniniPoggesi} for other related quantitative results.

\subsection{The main theorem}
The uniqueness problem $(UP_k) $, for $ k \geq 2 $ and without assuming uniform bounds on the interior or exterior touching balls, is a natural and interesting problem, which is to author's knowledge completely open, even in the $ 3 $-dimensional Euclidean space for sequences with vanishing oscillation of the Gaussian curvature (in some $ L^p $-norm). 
In this paper we study this problem for sequences of \emph{finite total curvature}, see Definition \ref{def finite total curvature}. As expected,  studying this problem under this new hypothesis requires the introduction of a substantial novel method of proof, with respect to the approaches used in Theorem \ref{positive reach} and in \cite{CiraoloRoncoroniVezzoni}, where the problem is treated under uniform bounds on the touching balls. Before to state the main result of the paper, firstly we introduce some definitions and notations. For a function $ f $ we write
$$ f^+ = \sup\{f,0\}\quad \textrm{and} \quad  f^- = -\inf\{f,0\}. $$
We denote with $ A_\Omega $ the norm of the second fundamental form of $ \partial\Omega $: 
\begin{equation*}
	A_\Omega(x) = \bigg(\sum_{i=1}^n \kappa_{\Omega, i}(x)^2 \bigg)^{\frac{1}{2}} \quad \textrm{for $ x \in \partial \Omega $.}
\end{equation*} 
\begin{definition}[Compactly supported sequences]\label{def compactly supported}
	We say that a sequence $ \Omega_\ell $ of subsets of $ \mathbb{R}^{n+1} $ is \emph{compactly supported} if there exists a ball $ B_R $ such that $ \Omega_\ell \subseteq B_R $ for all $ \ell \geq 1 $.
\end{definition}
\begin{definition}[Sequences of finite total curvature]\label{def finite total curvature}
	We say that a sequence $ \Omega_\ell \subseteq \mathbb{R}^{n+1} $ of open sets with $C^2 $-boundary has \emph{finite total curvature} if 
	\begin{equation*}
		\sup_{\ell \geq 1} \int_{\partial \Omega_\ell} A_{\Omega_\ell}^n\, d\mathcal{H}^n < \infty.
	\end{equation*}
\end{definition}
\begin{definition}[Asymptotically $ k $-mean convex sequences]\label{def asymptotically k convex}
	Let $ k \in \{1, \ldots, n\} $. We say that a sequence $\Omega_\ell \subseteq \mathbb{R}^{n+1} $, $ \ell \geq 1 $, of open sets with $ C^2 $-boundary is \emph{asymptotically $ k $-mean convex} if and only if 
	$$ \lim_{\ell \to \infty}\int_{\partial \Omega_\ell} (H_{\Omega_\ell, i})^- \, d\mathcal{H} ^n = 0 \quad \textrm{for $ i = 1, \ldots , k $.} $$
	
\noindent If $ k = 1 $ we simply say that $ \Omega_\ell $ is asymptotically mean convex.
\end{definition}

\begin{remark}\label{rmk mean convexity}
In relation to Definition \ref{def asymptotically k convex}, we recall a well known fact.	If $ k \in \{1, \ldots , n\} $ and $ \Omega \subseteq \mathbb{R}^{n+1} $ is a bounded and connected $C^2 $-domain with constant $ k $-th mean curvature function, then 
\begin{equation}\label{rmk mean convexity eq1}
	H_{\Omega, i}(x) > 0 \quad \textrm{for every $ x \in \partial \Omega $ and for $ 1 \leq i \leq k $.}
\end{equation}  
	 See \cite[page 450]{RosRevista} for details. The condition \eqref{rmk mean convexity eq1} is usually called $ k $-convexity or $ k $-mean convexity. It naturally appears in many problems involving higher-order mean curvature functions as it guarantees the ellipticity of the related fully non-linear PDE's equations; see the pioneering \cite{CaffarelliNirenberSpruck85}.
\end{remark}

This is the main result of the paper.
\begin{theorem}\label{main}
	Let $ n \geq 2 $, $ k \in \{2, \ldots , n\} $ and let $ \Omega_\ell \subseteq \mathbb{R}^{n+1} $ be a compactly supported and asymptotically $ (k-1) $-mean convex sequence with finite total curvature. Suppose that the sets $ \Omega_\ell $ converge in volume and perimeter to a set $ \Omega \subseteq \mathbb{R}^{n+1}$, and  there exist $ \lambda \in \mathbb{R} $ and $ M > 0 $ such that
	\begin{equation}\label{main: hp1}
		\lim_{\ell \to \infty}\int_{\partial \Omega_\ell}| H_{\Omega_\ell, k} -\lambda | \, d\mathcal{H}^n =0
	\end{equation} 
		\begin{equation}\label{main: hp3}
		\lim_{\ell \to \infty} \mathcal{H}^n(\{x \in \partial \Omega_\ell: H_{\Omega_\ell, 1}(x) \geq M\})=0.
	\end{equation}

Then $ \Omega $ is $ \mathcal{L}^{n+1} $ almost equal to a finite union of closed balls of the same radius with disjointed interiors.  The radius $ \rho $ of the balls satisfies the relations
	\begin{equation}\label{main: conclusion}
		\rho = \frac{(n+1)\mathcal{L}^{n+1}(\Omega)}{\mathcal{H}^n(\partial^\ast \Omega)} \quad \textrm{and} \quad \lambda = {n \choose k}\rho^{-k}.
	\end{equation}
\end{theorem}

\noindent In the classical case of $ 2 $-dimensional surfaces in $ \mathbb{R}^3 $,  in Theorem \ref{main} (i.e.\ $ n = 2 $ and $ k = 2 $) one can equivalently replace the condition of finite total curvature with a uniform bound on the $ L^2 $-norm of the mean curvature functions, since $A_{\Omega_\ell}^2 = H_{\Omega_\ell, 1}^2 - 2 H_{\Omega_\ell, 2} $.

\begin{corollary}\label{main corollary}
	Let $ \Omega_\ell \subseteq \mathbb{R}^{3} $ be a compactly supported and asymptotically mean convex sequence with
	$$ \sup_{\ell \geq 1} \int_{\partial \Omega_\ell} H_{\Omega_\ell, 1}^2\, d\mathcal{H}^2 < \infty. $$
	Suppose that the sets $ \Omega_\ell $ converge in volume and perimeter to a set $ \Omega \subseteq \mathbb{R}^{3}$,  and  there exist $ \lambda \in \mathbb{R} $ and $ M > 0 $ such that
	\begin{equation*}
		\lim_{\ell \to \infty}\int_{\partial \Omega_\ell}| H_{\Omega_\ell, 2} -\lambda | \, d\mathcal{H}^2 =0
	\end{equation*} 
	\begin{equation}\label{main corollary: hp}
		\lim_{\ell \to \infty} \mathcal{H}^2(\{x \in \partial \Omega_\ell: H_{\Omega_\ell, 1}(x) \geq M\})=0.
	\end{equation}
	
	Then the conclusion of Theorem \ref{main} holds with $ n = k = 2 $.
\end{corollary}

\begin{remark}
The hypothesis \eqref{main: hp3} in Theorem \ref{main} (or \eqref{main corollary: hp} in Corollary \ref{main corollary})	amounts to require that the mean curvatures are uniformly bounded from above outside a set of vanishing area. This is a  technical condition, which is used in section \ref{section HK} to derive some key fine properties of the varifold associated with the reduced boundary of the limiting set $ \Omega $. This condition is clearly satisfied whenever one has a sequence such that each set $ \Omega_\ell $ is made of several domains of uniformly bounded mean curvature connected by small necks of vanishing area (in which case we do not need to care about the behaviour of the mean curvatures on the vanishing necks). 

We conjecture that the hypothesis \eqref{main: hp3}  can be completely removed from Theorem \ref{main}.
\end{remark}

\begin{remark}
Besides its own interest, the uniqueness problem $(UP_k)$ for $ k \geq 2 $ naturally arises in the study of higher-order isoperimetric type inequalities (see \cite{GuanLi2009} and \cite{ChangWang2014}) and in the study of the geometric properties of hypersurfaces with prescribed higher-order mean curvatures (see \cite{GuanLiLi} and references therein).
\end{remark}

\subsection{Method of proof and organization of the paper}
The proof of Theorem \ref{main} is based on a novel geometric-measure theoretic extension of the method pioneered by Montiel and Ros. In \cite{RosRevista} and  \cite{MontielRos} the proof is based on two fundamental steps:  (1) proving a sharp Heintze-Karcher inequality with equality achieved only by balls, (2) checking that a domain with constant $ k $-th mean curvature realizes the equality case. This last step for $ k \geq 2 $ is crucially based on the Minkowski identities in \cite{Hsiung}, while it is an immediate consequence of divergence theorem for $ k = 1 $. 

The reduced boundary of the limiting set $ \Omega $ in Theorem \ref{main} is a $ n $-dimensional varifold $ V_\Omega $ of bounded mean curvature in $ \mathbb{R}^{n+1} $. In section \ref{section HK}, combining results in varifolds theory (\cite{Menne13} and \cite{SantilliBulletin}) with fine properties of the curvature  for arbitrary closed sets (\cite{SantilliAnnali}), we can quickly deduce from the general Heintze-Karcher inequality in \cite[Theorem 3.20]{HugSantilli} a sharp geometric inequality for sets of finite perimeter and bounded distributional mean curvature; see Theorem \ref{HK}.  This completes the first step of the proof. In order to complete the proof, we have to check that the limiting set $ \Omega $ satisfies the equality case in Theorem \ref{HK}, and this is the key new difficulty with respect to other aforementioned results based on the Montiel-Ros method. In fact, one needs to use a  Minkowski-type formula for the limiting set in order to use the geometric information given by the vanishing oscillation hypothesis \eqref{main: hp1} within the geometric inequality from Theorem \ref{HK}. However, Minkowski formulae for singular geometric sets are known only in very special cases, namely sets of positive reach and subanalytic sets, and their proof is a quite subtle issue based on the existence of a normal cycle; see \cite[section 3]{Fu98}. In particular, no Minkowski formulae are known in the varifolds setting. To deal with this point in Theorem \ref{main}, we consider the sequence of normal cycles $ N_{\Omega_\ell} $ associated with the exterior unit-normal bundles of the sets $ \Omega_\ell $. These are $ n $-dimensional integral currents in the product space $ \mathbb{R}^{n+1}\times \mathbb{S}^n $, which are cycles (i.e.\ $ \partial N_{\Omega_\ell} =0 $ in the sense of currents) and satisfy a Legendrian-type property, see Remark \ref{rmk: normal cycle C2 boundaries}.  It follows from the finite total curvature assumption that the masses of the integral currents $N_{\Omega_\ell} $ are uniformly bounded; henceforth we can apply Federer-Fleming compactness theorem to find, up to subsequences, that the currents $ N_{\Omega_\ell} $ converge weakly to a Legendrian cycle $ T $; see section \ref{section Legendrian cycles}. Using a representation formula for the curvature measures associated with $ T $ (see Lemma \ref{lem: representation curvature measures}), we pass to limit in the Minkowski formulae for $ \Omega_\ell $ to find a Minkowski-type formula for the Legendrian cycle $ T $; see Lemma \ref{lem: Minkowski formulae}. Now it is still not clear how to use the Minkowski formulae for $ T $ within the Heintze-Karcher inequality for $ V_\Omega $, since $ V_\Omega $ and $ T $  come from two completely different limit procedures: $ T $ is the limit in the sense of currents of the normal cycles of $ \Omega_\ell $, while $ V_\Omega $ is the limit in the sense of varifolds of the boundaries of $ \Omega_\ell $. This is a quite subtle point  and occupies most of the proof of Theorem \ref{main} in section \ref{section main}.

Finally we mention that in the appendix, firstly we generalize the soap bubble theorem \cite[Theorem A]{HugSantilli} (allowing the constant $ \lambda \in \mathbb{R} $), then we show how this theorem allows to characterize the limit of sequences of $C^2$-domains with almost constant $ k $-th mean curvature function and with a uniform bound on the exterior touching balls.

\medskip

\noindent {\bf Acknowledgements:} The author wishes to thank Francesco Maggi for useful comments on a preliminary version of this work. The author is partially supported by INdAM-GNSAGA.

\section{Preliminaries}\label{section preliminaries}

Let $ \pi_0 : \mathbb{R}^{n+1}\times \mathbb{R}^{n+1}\rightarrow \mathbb{R}^{n+1} $ and $ \pi_1 : \mathbb{R}^{n+1}\times \mathbb{R}^{n+1}\rightarrow \mathbb{R}^{n+1} $ be  defined as
$$ \pi_0(x,u) = x, \qquad \pi_1(x,u) = u. $$
If $ S \subseteq \mathbb{R}^{n+1} $ we define $ \dist(x, S) = \inf\{|x-a|: a \in S\} $ for $ x \in \mathbb{R}^{n+1} $. If $ Q \subseteq \mathbb{R}^{n+1}\times \mathbb{R}^{n+1} $ and $ S \subseteq \mathbb{R}^{n+1} $ we define 
$$ Q \restrict S = \pi_0^{-1}(S) \cap Q. $$

We say that a subset $ S \subseteq \mathbb{R}^m $ is \emph{countably $ \mathcal{H}^k $-rectifiable} if there exists a countable family $ F $ of $ k $-dimensional embedded $ C^1 $-submanifolds of $ \mathbb{R}^m $ such that $\mathcal{H}^k(S \setminus \bigcup F) = 0 $. We denote with $ \Tan^k(\mathcal{H}^k\restrict S, x) $ the \emph{$ \mathcal{H}^k $-approximate tangent cone} of $ S $ at $ x $; see \cite[3.2.16]{Fed69}. If $ \mathcal{H}^k(S) < \infty $, $ f  $ is a Lipschitz function on $ S $, then $ \Tan^k(\mathcal{H}^k\restrict S, x) $ is a  $ k $-dimensional plane at $ \mathcal{H}^k $ a.e.\ $ x \in S $ and we denote with $ J_k^Sf $ the $ k $-dimensional approximate tangential jacobian function of $ f $; see  \cite[3.2.19, 3.2.20]{Fed69}.

\subsection{Sets of finite perimeter} We refer to \cite[Chapter 3]{AFP00} or \cite{Maggibook} for details. We recall that  $ \Omega \subseteq \mathbb{R}^{n+1} $ is a set of finite perimeter in $ \mathbb{R}^{n+1} $ if its characteristic function $ \bm{1}_\Omega $ is a function of bounded variation in $ \mathbb{R}^{n+1} $. The reduced boundary $ \partial^\ast \Omega $ of $ \Omega $ is the set of points $ x \in \mathbb{R}^{n+1} $ such that the following limit  
$$ \lim_{r \to 0} \frac{\Der \bm{1}_\Omega(B_r(x))}{|\Der \bm{1}_\Omega(B_r(x))|}  $$
exists and belongs to $ \mathbb{S}^n $, in which case we denote it by  $ \nu_\Omega(x) $ (here $ \Der \bm{1}_\Omega $ is the distributional gradient of $ \bm{1}_\Omega $). The reduced boundary is countably $ \mathcal{H}^n $-rectifiable and $ \mathcal{H}^n(\partial^\ast \Omega)  $ equals the total variation of $ \Der \bm{1}_\Omega $; this number is the perimeter of $ \Omega $. The map $ \nu_\Omega:  \partial^\ast \Omega \rightarrow \mathbb{S}^n $ is the measure-theoretic exterior unit-normal of $ \Omega $ and we define
$$ \overline{\nu}_\Omega : \partial^\ast \Omega \rightarrow \mathbb{R}^{n+1} \times \mathbb{S}^n , \qquad  \overline{\nu}_\Omega(x) =(x, \nu_\Omega(x)). $$
Notice that if $ \Omega $ is a $ C^1 $-domain then $ \partial^\ast \Omega = \partial \Omega $ and $ \nu_\Omega $ is the classical exterior unit-normal vector field of $ \Omega $.
\subsection{Currents} We refer to \cite[Chapter 6]{Simonbook} for details. The space of compactly supported $k$-forms on an open subset $ U $ of $ \mathbb{R}^p $ is denoted by the usual $ \mathcal{D}^k(U) $ and the space of $ k $-currents on $ U $ by $ \mathcal{D}_k(U) $. We denote with $ \bm{M}_W(T) $ the mass of $ T $ over an open subset $ W $ of $ U $. A sequence $ T_\ell \in \mathcal{D}_k(U)$ \emph{weakly converges} to $ T\in \mathcal{D}_k(U) $ if and only if
$$ T_\ell(\phi) \to T(\phi)\qquad \textrm{for all $ \phi \in \mathcal{D}^k(U) $.} $$

We say that a $ k $-current $ T \in \mathcal{D}_k(U) $ is an \emph{interger multiplicity rectifiable $ k $-current of $ U $} provided
$$ T(\phi) = \int_{M}\langle \eta(x), \phi(x) \rangle \theta(x)\, d\mathcal{H}^k(x) \qquad \textrm{for all $ \phi \in \mathcal{D}^k(U) $,} $$
where $ M $ is a countably $ \mathcal{H}^k $-rectifiable subset of $ U $, $ \theta : M \rightarrow \mathbb{Z}_+ $ is an $ \mathcal{H}^k $-measurable function such that $ \int_{K \cap M} \theta\, d\mathcal{H}^k < \infty $ for every compact subset $ K $ of $ U $, and $ \eta(x) = \tau_1(x) \wedge \ldots \wedge \tau_k(x) $ for $ \mathcal{H}^k $ a.e.\ $ x \in M $, where $ \tau_1(x) , \ldots , \tau_k(x) $ form an orthonormal basis of $ \Tan^k(\mathcal{H}^k \restrict M,x) $ for $ \mathcal{H}^k $ a.e.\ $ x \in M $. The set $ M $ is called \emph{carrier} of $T$ and it is $ \mathcal{H}^n $ almost uniquely determined by $ T $. Henceforth for each integer multiplicity rectifiable $ k $-current $ T $ we introduce the symbol $ W_T $ for the carrier of $ T $.
One can easily check that if $ W \subseteq U $ is an open set then   
$$ \bm{M}_W(T) = \int_{W \cap M}\theta\, d\mathcal{H}^k. $$

\subsection{Varifolds} We refer to \cite{Allard72} for details. Let $ U \subseteq \mathbb{R}^{n+1} $ be an open set. The space of $ k $-dimensional varifolds on $ U $ is denoted with $ \bm{V}_k(U) $ and the space of $ k $-dimensional integral varifolds on $ U $ with the usual $ \bm{IV}_k(U) $. Associated with $ V \in \bm{V}_k(U) $ we consider the weight measure $ \| V \| $, which is the Radon measure on $ U $ defined by $ \| V \|(S) = V(S \times \bm{G}(n+1,k)) $ for every $ S \subseteq U $,  and we denote the \emph{first variation of $ V $}  by $ \delta V $; see \cite[section 4]{Allard72} for details. A varifold $ V \in \bm{V}_k(U) $ is a varifold of \emph{bounded mean curvature} if there exists a function $ h \in L^\infty(\| V \|, \mathbb{R}^{n+1}) $ so that 
$$ \delta V(X) = \int h(x) \bullet X(x)\, d\| V \|(x), \quad \textrm{for all $ X \in \mathcal{C}^\infty_c(U, \mathbb{R}^{n+1}) $.} $$
The function $ h $ is uniquely determined by $ V $ and we write $ h = \bm{h}(V,\cdot) $.

In this paper we consider integral varifolds $ V \in \bm{IV}_n(\mathbb{R}^{n+1}) $ associated with the reduced boundary of a set of finite perimeter; namely if $ E $ is a set of finite perimeter in $ \mathbb{R}^{n+1}$ of positive volume, we define the  $ n $-dimensional  varifold $ V_E \in \bm{IV}_n(\mathbb{R}^{n+1})$ as the unique Radon measure $ V_E $ on $ \mathbb{R}^{n+1}\times \bm{G}(n+1,n) $ such that
$$ \int \phi(x, S)\, dV_E(x,S) = \int_{\partial^\ast E} \phi(x, \Tan^n(\mathcal{H}^n \restrict \partial^\ast E,x))\, d\mathcal{H}^n(x) $$
for every $ \phi \in  C_c(\mathbb{R}^{n+1}\times \bm{G}(n+1,n)) $.
 Notice that $ \| V_E \| = \mathcal{H}^n \restrict \partial^\ast E $ and
\begin{equation}\label{first variation eq1}
	\delta V_E(X) = \int_{\partial^\ast E} {\rm div}_{\mathbb{R}^{n+1}}(X)\, d\mathcal{H}^n - \int_{\partial^\ast E} \Der X(x)(\nu_E(x)) \bullet \nu_E(x)\, d\mathcal{H}^n(x)
\end{equation} 
for all $ X \in \mathcal{C}^\infty_c(\mathbb{R}^{n+1}, \mathbb{R}^{n+1}) $. 

\subsection{Normal bundle of closed sets} Suppose $ C \subseteq \mathbb{R}^{n+1} $ is a closed set. We define 
$$ \nor(C) = \{(x,u)\in C \times \mathbb{S}^{n}: \dist(x+su, C) = s \; \textrm{for some $ s > 0 $}\}. $$
and we recall that $ \nor(C) $ is always countably $ \mathcal{H}^n $-rectifiable; see \cite{SantilliAnnali}\footnote{The unit normal bundle of a closed set $ C $ in \cite{SantilliAnnali} is denoted with $N(C)$.}. On the other hand we remark that there are closed sets $ C $ for which $ \nor(C)  $ has not locally finite $ \mathcal{H}^n $-measure;  indeed this can happen even when $ C = \spt \| V \| $ and $ V \in \bm{IV}_2(\mathbb{R}^3) $ is a varifold of bounded mean curvature.

It is proved in \cite{SantilliAnnali} that for $ \mathcal{H}^n $ a.e.\ $(x,u)\in \nor(C) $ there exist a linear subspace $ T_C(x,u) $ of $ \mathbb{R}^{n+1} $ and a symmetric bilinear form $ Q_C(x,u) : T_C(x,u) \times T_C(x,u) \rightarrow \mathbb{R} $, whose eigenvalues can be used to provide an explicit representation of the approximate tangent space of $ \nor(C) $ at $ \mathcal{H}^n $ almost all points. To this end, for $ \mathcal{H}^n $ a.e.\ $(x,u)\in \nor(C) $, we define $$ -\infty < \kappa_{C,1}(x,u)\leq \ldots \leq \kappa_{C,n}(x,u) \leq \infty $$ in the following way: $ \kappa_{C,1}(x,u), \ldots , \kappa_{C,m}(x,u) $ are the eigenvalues of $ Q_C(x,u) $, where $ m = \dim T_C(x,u) $, and $ \kappa_{C,i}(x,u) = +\infty $ for $ i = m+1, \ldots , n $. Given this definition we recall the following lemma.

\begin{lemma}[\protect{cf.\ \cite[Lemma 4.11]{SantilliAnnali}}]\label{lem: Santilli20}
Let $ Q \subseteq \nor(C) $ be an $ \mathcal{H}^n $-measurable set with finite $ \mathcal{H}^n $ measure.	For $ \mathcal{H}^n $ a.e.\ $(x,u) \in Q $ there exist $ \tau_1(x,u), \ldots , \tau_n(x,u)\in \mathbb{R}^{n+1} $ such that $ \{\tau_1(x,u), \ldots, \tau_n(x,u), u\} $ is an orthonormal basis of $ \mathbb{R}^{n+1} $ and  $ \zeta_1(x, u), \ldots , \zeta_n(x,u) \in \mathbb{R}^{n+1}\times \mathbb{R}^{n+1} $, defined as
$$ \zeta_i(x,u) = (\tau_i(x,u), \kappa_i(x,u)\tau_i(x,u)) \quad \textrm{if $ \kappa_{C,i}(x,u)< \infty $} $$
$$ \zeta_i(x,u) = (0,\tau_i(x,u)) \quad \textrm{if $ \kappa_{C,i}(x,u)= \infty $}, $$
form an orthogonal basis of $ \Tan^n(\mathcal{H}^n \restrict Q, (x,u)) $.
	\end{lemma}

We also need the following result, relating the numbers $ \kappa_{C,i} $ with the principal curvature of a $ C^2 $-hypersurface that intersects $ C $. 

\begin{lemma}[\protect{cf.\ \cite[Lemma 6.1]{SantilliAnnali}}]\label{lem: Santilli20II}
If $ \Sigma $ is an embedded $ C^2 $-hypersurface then there exists $ R \subseteq \Sigma \cap \partial C $ such that $$ \mathcal{H}^n((\Sigma \cap \partial C) \setminus R) =0, \qquad  \nor(C)\restrict R \subseteq \nor(\Sigma) $$ and 
\begin{equation*}
	\kappa_{C,i}(x,u) = \kappa_{\Sigma,i}(x,u) \quad \textrm{for $ \mathcal{H}^n $ a.e.\ $(x,u) \in \nor(C)\restrict R $}
\end{equation*}  
for $ 1 \leq i \leq n $. Here $ \nor(\Sigma) $ is the classical unit-normal bundle of $ \Sigma $ and $$ \kappa_{\Sigma,1}(x,u)\leq \ldots \leq \kappa_{\Sigma, n}(x,u) $$ are the principal curvatures of $ \Sigma $ at $ x $ in the direction $ u $.
\end{lemma}

We recall now the general Heintze-Karcher type inequality for arbitrary closed sets proved in \cite{HugSantilli}, which is used in the proof of Theorem \ref{HK}. 

\begin{theorem}[\protect{cf.\ \cite[Theorem 3.20]{HugSantilli}}]\label{HK general}
	Let $ C \subseteq \mathbb{R}^{n+1} $ be a bounded closed set with non empty interior. Let $ K = \mathbb{R}^{n+1}  \setminus {\rm interior}(C) $ and assume that 
	$$ \sum_{i=1}^n \kappa_{K,i}(x,u) \leq 0 \quad \textrm{for $ \mathcal{H}^n $ a.e.\ $(x,u) \in \nor(K)$.} $$
	Then 
	$$ (n+1)\mathcal{L}^{n+1}({\rm interior}(C)) \leq  \int_{\nor(K)}J_n^{\nor(K)}\pi_0(x,u)\,\frac{n}{|\sum_{i=1}^n \kappa_{K,i}(x,u)|}\, d\mathcal{H}^n(x,u). $$
	It the equality holds and there exists $ q < \infty $ such that $|\sum_{i=1}^n \kappa_{K,i}(x,u)|\leq q $ for $ \mathcal{H}^n $ a.e.\ $(x,u)\in \nor(K) $, then $ C $ equals the union of finitely many closed balls with disjointed interiors.
\end{theorem}

\section{Legendrian cycles and Curvature measures}\label{section Legendrian cycles}

We denote with $ E_{n+1}\in \mathcal{D}^{n+1}(\mathbb{R}^{n+1}) $ the \emph{canonical volume form} of $ \mathbb{R}^{n+1} $, and with $$ \alpha_{\mathbb{R}^{n+1}} : \mathbb{R}^{n+1}\times \mathbb{R}^{n+1} \rightarrow {\textstyle \bigwedge^{1}}(\mathbb{R}^{n+1}\times \mathbb{R}^{n+1}) $$ the \emph{contact $ 1 $-form}; i.e.\
$$ \langle (y,v), \alpha_{\mathbb{R}^{n+1}}(x,u)\rangle = y \bullet u \qquad \textrm{for $ (x,u), (y,v)\in \mathbb{R}^{n+1} $.} $$

A \emph{Legendrian cycle of $ \mathbb{R}^{n+1} $} is an integer-multiplicity rectifiable $ n $-current $ T $ of $ \mathbb{R}^{n+1} \times \mathbb{R}^{n+1} $ such that 
\begin{equation*}
	\textrm{$\spt(T) $ is compact subset of $\mathbb{R}^{n+1} \times \mathbb{S}^n$,} \quad \partial T =0, \quad T \restrict \alpha_{\mathbb{R}^{n+1}} =0.
\end{equation*}

The following result describes the approximate tangent space of the carrier of a Legendrian cycle.
\begin{theorem}[\protect{cf.\ \cite[Theorem 9.2]{RatajZaehlebook}}]\label{theo: Rataj-Zaehle}
	Let $ T $ be a Legendrian cycle of $ \mathbb{R}^{n+1} $ with carrier $ W_T $. For $ \mathcal{H}^n $ a.e.\ $ (x,u)\in W_T $  there exist numbers $ -\infty < \kappa_1(x,u) \leq \ldots \leq \kappa_n(x,u) \leq \infty $ and vectors $ \tau_1(x,u), \ldots ,\tau_n(x,u) $ such that  $ \{\tau_1(x,u), \ldots , \tau_n(x,u), u\}  $ is a positively oriented basis of $ \mathbb{R}^{n+1} $ (i.e.\ $ \langle \tau_1(x,u) \wedge \ldots \wedge \tau_n(x,u)\wedge u , E_{n+1}\rangle = 1 $)
	and the vectors $ \zeta_1(x, u), \ldots , \zeta_n(x,u)\in \mathbb{R}^{n+1}\times \mathbb{R}^{n+1} $, defined as
	$$ \zeta_i(x,u) = (\tau_i(x,u), \kappa_i(x,u)\tau_i(x,u)) \quad \textrm{if $ \kappa_i(x,u)< \infty $} $$
	$$ \zeta_i(x,u) = (0,\tau_i(x,u)) \quad \textrm{if $ \kappa_i(x,u)= \infty $}, $$
	form an orthogonal basis of $ \Tan^n(\mathcal{H}^n \restrict W_T, (x,u)) $. The numbers $ \kappa_i(x,u) $ for $ i = 1, \ldots , n $ are uniquely determined, as well as the subspaces of $ \mathbb{R}^{n+1} $ spanned by $ \{ \tau_i(x,u): \pi_1(\zeta_i(x,u)) = \kappa_k(x,u)\tau_i(x,u)\} $ for each $ k $.
\end{theorem}

\begin{definition}
	Given $ T $, $ \tau_1, \ldots , \tau_n $ and $ \zeta_1, \ldots , \zeta_n $ as in Theorem \ref{theo: Rataj-Zaehle}, we define 
	$$ \mathcal{\mathcal{J}}_T(x,u) = \frac{1}{| \zeta_1(x,u)\wedge \ldots \wedge \zeta_n(x,u)|} $$
	and 
	$$ \zeta_T(x,u) = \frac{\zeta_1(x,u)\wedge \ldots \wedge \zeta_n(x,u)}{|\zeta_1(x,u)\wedge \ldots \wedge \zeta_n(x,u)|} $$
	for $ \mathcal{H}^n $ a.e.\ $ (x,u) \in W_T $.
\end{definition}

\begin{remark}
This definition is well-posed, in the sense that $ \mathcal{J}_T $ and $ \zeta_T $  do not depend, up to a set of $ \mathcal{H}^n $-measure zero, on the choice of the vectors $ \tau_1, \ldots, \tau_n $.
\end{remark}

\begin{remark}
There exists a $ \mathcal{H}^n\restrict W_T $ almost unique integer-valued map $ i_T $ such that $ i_T(x,u) =0 $ for $ \mathcal{H}^n $ a.e.\ $ (x,u)\in \mathbb{R}^{n+1}\times \mathbb{R}^{n+1} \setminus W_T $, $ i_T(x,u) > 0 $ for $ \mathcal{H}^n $ a.e.\ $(x,u)\in W_T $ and 
$$ T = (\mathcal{H}^n \restrict W_T)i_T \wedge \zeta_T. $$
\end{remark}

\begin{definition}[Lipschitz-Killing differential forms]
	For each $ k \in \{0, \ldots , n\} $ we define
	$$ \Sigma_{n,k} = \bigg\{ \sigma: \{1, \ldots , n\} \rightarrow \{0,1\}: \sum_{i=1}^n \sigma(i) = n-k \bigg\}. $$
	The \emph{$ k $-th Lipschitz-Killing differential form} $ \varphi_k : \mathbb{R}^{n+1} \times \mathbb{R}^{n+1} \rightarrow \bigwedge^n(\mathbb{R}^{n+1} \times \mathbb{R}^{n+1}) $  is defined by
	\begin{flalign*}
		&\langle \xi_1 \wedge \cdots \wedge \xi_n, \varphi_k(x,u) \rangle\\
		& \qquad = \sum_{\sigma \in \Sigma_{n,k}} \langle \pi_{\sigma(1)}(\xi_1)\wedge \cdots \wedge \pi_{\sigma(n)}(\xi_n) \wedge u, E_{n+1}\rangle,
	\end{flalign*} 
	for every $ \xi_1, \ldots , \xi_n \in \mathbb{R}^{n+1}\times \mathbb{R}^{n+1}$.
\end{definition}

\begin{definition}[Curvature measures]
	If $ T $ is a Legendrian cycle in $ \mathbb{R}^{n+1} $ and $ k \in \{0, \ldots , n \} $, then \emph{$ k $-th curvature measure}  of $ T $ is the $ 0 $-current   $$ T \restrict \varphi_k \in \mathcal{D}_0(\mathbb{R}^{n+1} \times \mathbb{R}^{n+1}). $$
\end{definition}

\begin{definition}[Principal curvatures]
Given a Legendrian cycle $ T $ of $ \mathbb{R}^{n+1} $ we define $ \kappa_{T,i} = \kappa_i $, where $ \kappa_1 \leq \ldots \leq \kappa_n $ are the functions defined $ \mathcal{H}^n $ a.e.\ on $ W_T $ given by Theorem \ref{theo: Rataj-Zaehle}. 

The numbers $ \kappa_{T,1} \leq \ldots \leq \kappa_{T,n} $ are the principal curvatures of $ T $.
\end{definition}
\begin{definition}[Mean curvature functions]
Given a Legendrian cycle $ T $ of $ \mathbb{R}^{n+1} $ we define 
$$ W_T^{(i)} = \{(x,u)\in W_T: \kappa_{T,i}(x,u)< \infty, \; \kappa_{T,i+1}(x,u) = +\infty\} \quad \textrm{for $ i = 1, \ldots , n-1 $} $$
$$ W_T^{(0)} = \{(x,u)\in W_T : \kappa_{T,1}(x,u) = +\infty\} $$
$$ W_T^{(n)} = \{(x,u)\in W_T : \kappa_{T,n}(x,u) < +\infty\}. $$

Then we define the \emph{$ r $-th mean curvature function of $T $} as 
$$ H_{T,r} =\bm{1}_{W_T^{(n-r)}} + \sum_{i=1}^r \sum_{\lambda \in \Lambda_{n-r+i,i}} \kappa_{T, \lambda(1)}\cdots \kappa_{T, \lambda(i)}\bm{1}_{W_T^{(n-r+i)}}  \qquad \textrm{for $ r =1, \ldots, n $} $$
where $ \Lambda_{n,k} $, for $ k \leq n $, is the set of all increasing functions from $ \{1, \ldots , k\} $ to $ \{1, \ldots , n\} $.  Additionally,
$$ H_{T,0} = \bm{1}_{W_T^{(n)}}. $$
\end{definition}

\begin{remark}\label{remark: jacodian of pi0}
	It follows from Theorem \ref{theo: Rataj-Zaehle} that $$ J_n^{W_T}\pi_0(x,u) = \mathcal{J}_T(x,u) >0 \quad \textrm{for $\mathcal{H}^n $ a.e.\ $ (x,u)\in W_T^{(n)} $}$$
	 and $$ J_n^{W_T}\pi_0(x,u) =0 \quad \textrm{for $ \mathcal{H}^n $ a.e.\ $(x,u)\in W_T \setminus W_T^{(n)} $}. $$
\end{remark}

The definition of mean curvature functions is clearly motivated by the following result.

\begin{lemma}\label{lem: representation curvature measures}
If $ T $ is a Legendrian cycle of $ \mathbb{R}^{n+1} $ and $ k \in \{0, \ldots, n\} $ then 
$$ (T \restrict \varphi_k)(\phi) = \int \phi(x,u)\, i_T(x,u)\, \mathcal{J}_T(x,u) \, H_{T,n-k}(x,u)\, d\mathcal{H}^n(x,u) $$
for every $ \phi \in \mathcal{D}^0(\mathbb{R}^{n+1}\times \mathbb{R}^{n+1}) $.
\end{lemma}

\begin{proof}
 For $ j \in \{0, \ldots , n\} $ we define $$ \Sigma_{n,k}^{(j)} = \{\sigma \in \Sigma_{n,k} : \sigma(i) = 1 \; \textrm{for every $ i > j $}\}. $$
Notice that $ \Sigma_{n,k}^{(j)} = \varnothing $ if $ j < k $. Moreover, 
$$ \langle \pi_{\sigma(1)}(\zeta_{1}(x,u)) \wedge \ldots \wedge \pi_{\sigma(n)}(\zeta_{n}(x,u))\wedge u, E_{n+1}\rangle =0 $$
for $ (x,u) \in W_T^{(j)} $ and $ \sigma \in \Sigma_{n,k} \setminus \Sigma_{n,k}^{(j)} $, where $ \tau_1, \ldots , \tau_n $ and $ \zeta_1, \ldots , \zeta_n $ are given as in Theorem \ref{theo: Rataj-Zaehle}. Therefore  we infer that 
\begin{flalign*}
	&(T \restrict \varphi_k)(\phi) \\
	& \quad = \sum_{j=k}^n\sum_{\sigma \in \Sigma_{n,k}^{(j)}} \int_{W_T^{(j)}}\phi(x,u)\,i_T(x,u)\, \mathcal{J}_T(x,u)\,\langle \pi_{\sigma(1)}(\zeta_{1}(x,u)) \wedge \ldots \\
	& \qquad \qquad \qquad \qquad  \ldots \wedge \pi_{\sigma(n)}(\zeta_{n}(x,u))\wedge u, E_{n+1}\rangle\, d\mathcal{H}^n(x,u).
\end{flalign*} 
Noting that $ \langle \tau_1(x,u)\wedge \ldots \wedge \tau_n(x,u)\wedge u, E_{n+1}\rangle = 1 $, we infer for $ j > k $ that
\begin{flalign*}
&\sum_{\sigma \in \Sigma_{n,k}^{(j)}} \int_{W_T^{(j)}}\phi(x,u)\, i_T(x,u)\, \mathcal{J}_T(x,u)\,\langle \pi_{\sigma(1)}(\zeta_{1}(x,u)) \wedge \ldots \\
& \qquad \qquad \qquad \qquad \ldots \wedge \pi_{\sigma(n)}(\zeta_{n}(x,u))\wedge u, E_{n+1}\rangle\, d\mathcal{H}^n(x,u) \\
& \quad\quad  = \sum_{\lambda \in \Lambda_{j,j-k}}\int_{W_T^{(j)}}\phi(x,u)\,i_T(x,u)\, \mathcal{J}_T(x,u) \prod_{\ell=1}^{j-k}\kappa_{T, \lambda(\ell)}(x,u)\, d\mathcal{H}^n(x,u),
\end{flalign*}
while for $ j = k $ we have that
\begin{flalign*}
	&\sum_{\sigma \in \Sigma_{n,k}^{(k)}} \int_{W_T^{(k)}}\phi(x,u)\,i_T(x,u)\, \mathcal{J}_T(x,u)\,\langle \pi_{\sigma(1)}(\zeta_{1}(x,u)) \wedge \ldots \\
	& \qquad \qquad \qquad \qquad \ldots  \wedge \pi_{\sigma(n)}(\zeta_{n}(x,u))\wedge u, E_{n+1}\rangle\, d\mathcal{H}^n(x,u)\\
	& \quad\quad  = \int_{W_T^{(k)}}\phi(x,u)\,i_T(x,u)\, \mathcal{J}_T(x,u) \, d\mathcal{H}^n(x,u).
\end{flalign*}
We conclude
\begin{flalign*}
&(T \restrict \varphi_k)(\phi) \\
& \quad = \sum_{i=1}^{n-k} \sum_{\lambda \in \Lambda_{k+i,i}}\int_{W_T^{(k+i)}}\phi(x,u)\,i_T(x,u)\, \mathcal{J}_T(x,u) \prod_{\ell=1}^{i}\kappa_{T, \lambda(\ell)}(x,u)\, d\mathcal{H}^n(x,u) \\
& \qquad \qquad  + \int_{W_T^{(k)}}\phi(x,u)\,i_T(x,u)\, \mathcal{J}_T(x,u) \, d\mathcal{H}^n(x,u),
\end{flalign*}  
which is the desired conclusion.
\end{proof}

We conclude this section by introducing the normal cycle of a $ C^2 $-domain. 

\begin{definition}[Normal cycle of a $C^2$-domain]
Suppose $ \Omega\subseteq \mathbb{R}^{n+1} $ is an open set with $C^2 $-boundary. Then we define 
$$ N_\Omega = (\overline{\nu}_\Omega)_{\#}((\mathcal{H}^n \restrict \partial \Omega) \wedge \star\, \nu_\Omega) \in \mathcal{D}_n(\mathbb{R}^{n+1} \times \mathbb{S}^n). $$
\end{definition}

\begin{remark}[Classical]\label{rmk: normal cycle C2 boundaries}
	Since $(\mathcal{H}^n \restrict \partial \Omega) \wedge \star\, \nu_\Omega $ is a $ n $-dimensional cycle of $ \mathbb{R}^{n+1} $, it follows that $ N_\Omega $ is a $ n $-dimensional cycle in $ \mathbb{R}^{n+1} \times \mathbb{S}^n $. Moreover, area formula for rectifiable currents (see \cite[4.1.30]{Fed69}) allows to conclude
	\begin{flalign}\label{normal cycle C2 boundaries eq1}
		N_\Omega(\psi) & = \int_{\partial \Omega}\langle {\textstyle \Lambda_n} \Der \overline{\nu}_\Omega(x)(\star \nu_\Omega(x)), \psi(\overline{\nu}_\Omega(x))\rangle\, d\mathcal{H}^n(x)\notag \\
		& =  \big[\big(\mathcal{H}^n \restrict \overline{\nu}_\Omega(\partial \Omega)\big) \wedge \eta\big](\psi), 
	\end{flalign}
	where
	$$  \eta(x,u) = \frac{{\textstyle \Lambda_n} \Der \overline{\nu}_\Omega(x)(\star \nu_\Omega(x))}{|{\textstyle \Lambda_n} \Der \overline{\nu}_\Omega(x)(\star \nu_\Omega(x))|} \quad \textrm{for $(x,u)\in \overline{\nu}_\Omega(\partial \Omega) $}. $$ It follows that $ N_\Omega \in \bm{I}_n(\mathbb{R}^{n+1}\times \mathbb{R}^{n+1}) $. We notice that 
	\begin{flalign*}
		\bm{M}(N_\Omega) = \mathcal{H}^n(\overline{\nu}_\Omega(\partial \Omega))  = \int_{\partial \Omega}J_n^{\partial \Omega}\overline{\nu}_\Omega\, d\mathcal{H}^n = \int_{\partial \Omega} \prod_{i=1}^n (1+ \kappa_{\Omega, i}^2)^{\frac{1}{2}}\, d\mathcal{H}^n,
	\end{flalign*} 
	and by arithmetic-geometric mean inequality we have that
\begin{flalign}\label{mass normal bundle}
	\bm{M}(N_\Omega) \leq 2^{\frac{n}{2}}\bigg(\mathcal{H}^n(\partial \Omega) + n^{-\frac{n}{2}}\int_{\Omega}A_\Omega(x)^n\, d\mathcal{H}^n\bigg).
	\end{flalign}
	 We check now that $ N_\Omega \restrict \alpha_{\mathbb{R}^{n+1}} =0 $: if $ x \in \partial \Omega $ and $ \tau_1, \ldots , \tau_n $ is an orthonormal basis of $ \Tan(\partial \Omega,x) $  such that $ \star \nu_\Omega = \tau_1 \wedge \ldots \wedge \tau_n $ and $ \Der \nu_\Omega(x)(\tau_i(x)) = \kappa_{\Omega, i}(x)\tau_i(x) $ for $ x \in \partial \Omega $,  then we use shuffle formula (see \cite[pag.\ 18]{Fed69}) to compute
	\begin{flalign*}
		&\langle {\textstyle \Lambda_n} \Der \overline{\nu}_\Omega(x)(\star \nu_\Omega(x)), (\psi \wedge \alpha_{\mathbb{R}^{n+1}})(\overline{\nu}_\Omega(x))\rangle \\
		& \qquad = \langle (\tau_1, \kappa_{\Omega,1}(x)\tau_1)\wedge \ldots \wedge (\tau_n, \kappa_{\Omega, n}(x)\tau_n), (\psi \wedge \alpha_{\mathbb{R}^{n+1}})(\overline{\nu}_\Omega(x)) \rangle =0
	\end{flalign*}
	for every $ \psi \in \mathcal{D}^{n-1}(\mathbb{R}^{n+1}\times \mathbb{R}^{n+1}) $. 
	We conclude that $ N_\Omega $ is a Legendrian cycle of $ \mathbb{R}^{n+1} $.
	 Additionally one easily infer that $ W_{N_\Omega}^{(n)} = W_{N_\Omega} = \overline{\nu}_\Omega(\partial \Omega) $ and  
	$$ \kappa_{N_\Omega,i}(x, u) = \kappa_{\Omega, i}(x), \quad  \mathcal{J}_{N_\Omega}(x,u) = \frac{1}{|{\textstyle \Lambda_n} \Der \overline{\nu}_\Omega(x)(\star \nu_\Omega(x))|}$$
	for every $(x,u)\in W_{N_\Omega} $.
	In particular, $ H_{N_{\Omega},k}(x,u) = H_{\Omega,k} (x) $ for  $ (x,u)\in \overline{\nu}_{\Omega}(\partial \Omega) $  and by Remark \ref{remark: jacodian of pi0}
	\begin{equation}\label{normal cycle C2 boundaries eq2}
(N_\Omega \restrict \varphi_{n-k})(\phi) = \int_{\partial \Omega}H_{\Omega, k}(x)\, \phi(x, \nu_\Omega(x))\, d\mathcal{H}^n(x)
	\end{equation}  
	for all $ \phi \in \mathcal{D}^0(\mathbb{R}^{n+1}\times \mathbb{R}^{n+1}) $ and $ k \in \{0, \ldots, n\} $.
\end{remark}

\begin{lemma}\label{lem asymptotically mean convexity}
	Let $ \Omega_\ell $ be an asymptotically $ k $-mean convex sequence such that $ N_{\Omega_\ell} $ weakly converges to a Legendrian cycle $ T $ of $ \mathbb{R}^{n+1} $. Then 
	$$ H_{T,i}(x,u) \geq 0 \quad \textrm{for $ \mathcal{H}^n $ a.e.\ $(x,u)\in W_T $ and for $ i \in \{1, \ldots , k\} $.} $$
\end{lemma}

\begin{proof}
Let $ i \in \{1, \ldots , k\} $. By \eqref{normal cycle C2 boundaries eq2} we have that
\begin{flalign*}
&\big(N_{\Omega_\ell}\restrict \varphi_{n-i}\big)(\phi) \\
& \qquad = -\int_{\partial \Omega_\ell} H_{\Omega_\ell, i}^-(x)\, \phi(x,\nu_{\Omega_\ell}(x))\, d\mathcal{H}^n(x)  +  \int_{\partial \Omega_\ell} H_{\Omega_\ell, i}^+(x)\, \phi(x,\nu_{\Omega_\ell}(x))\, d\mathcal{H}^n(x)
\end{flalign*}  
for  $ \phi \in \mathcal{D}^0(\mathbb{R}^{n+1}\times \mathbb{R}^{n+1}) $, whence we infer
$$ (T \restrict \varphi_{n-i})(\phi) = \lim_{i \to \infty} \int_{\partial \Omega_\ell} H_{\Omega_\ell, i}^+(x)\, \phi(x,\nu_{\Omega_\ell}(x))\, d\mathcal{H}^n(x) \geq 0 $$
for $ \phi \in \mathcal{D}^0(\mathbb{R}^{n+1}\times \mathbb{R}^{n+1}) $ with $ \phi \geq 0 $. Now we obtain the desired conclusion from Lemma \ref{lem: representation curvature measures}.
\end{proof}

\begin{lemma}[Minkowski-Hsiung identities]\label{lem: Minkowski formulae}
	Suppose $\Omega_\ell \subseteq \mathbb{R}^{n+1}$ is a compactly supported sequence of open sets with $ C^2 $-boundary and $ T $ is a Legendrian cycle of $ \mathbb{R}^{n+1} $ such that $N_{\Omega_\ell} \rightarrow T $ weakly.
	 
	Then the Minkowski identity holds for $ T $:
	\begin{flalign}\label{lem: Minkowski formulaeI}
		& (n-k+1)\int i_T(x,u)\, \mathcal{J}_T(x,u)\, H_{T,k-1}(x,u)\, d\mathcal{H}^n(x,u) \notag \\
		& \qquad  \qquad  = k \int (x \bullet u)i_T(x,u)\, \mathcal{J}_T(x,u)\, H_{T,k}(x,u)\, d\mathcal{H}^n(x,u)
	\end{flalign} 
for $ k = 1,\ldots , n $.	Additionally, if $ \Omega_\ell $ converges in measure to a set $\Omega$, then
	\begin{equation}\label{lem: Minkowski formulaeII}
		(n+1)\mathcal{L}^{n+1}(\Omega) = \int (x \bullet u)i_T(x,u)\, \mathcal{J}_T(x,u)\, H_{T,0}(x,u)\, d\mathcal{H}^n(x,u).
	\end{equation}  
\end{lemma}

\begin{proof}
Let  $ R > 0 $ such that $ \Omega_\ell \subseteq B_R $ for all $ \ell \geq 1 $. Define $ \sigma : \mathbb{R}^{n+1}\times \mathbb{S}^n \rightarrow \mathbb{R} $  as $ \sigma(x,u) = x \bullet u $ for $ (x,u)\in \mathbb{R}^{n+1}\times \mathbb{S}^n $ and choose  $ \phi \in\mathcal{D}^0(\mathbb{R}^{n+1}\times \mathbb{R}^{n+1}) $ such that $ \phi(x,u) =1 $ for all $ (x,u) \in B_{2R} $. The classical Minkowski-Hsiung identities for $ C^2 $-domains and \eqref{normal cycle C2 boundaries eq2} imply
	\begin{flalign*}
		 (n-k+1)N_{\Omega_\ell}(\phi\,\varphi_{n-k+1}) & = (n-k+1) \int_{\partial \Omega_\ell} H_{\Omega_\ell,k-1}(x)\, d\mathcal{H}^n(x)\notag \\
		& = k \int_{\partial \Omega_\ell}(x \bullet \nu_{\Omega_\ell}(x) )  H_{\Omega_\ell, k}(x)\, d\mathcal{H}^n(x) \notag  \\
		& = k N_{\Omega_\ell}(\sigma\, \phi\, \varphi_{n-k}).
	\end{flalign*}
If $ \ell \to \infty $ this yields that 
$$ (n-k+1)T(\phi \,\varphi_{n-k+1})  = k T(\sigma\, \phi\, \varphi_{n-k}) $$
and  we obtain \eqref{lem: Minkowski formulaeI} from Lemma \ref{lem: representation curvature measures}. Analogously, if $ \Omega_\ell \to \Omega $ in measure, then divergence theorem and area formula imply
$$
		(n+1)\mathcal{L}^{n+1}(\Omega_\ell)  = \int_{\partial \Omega_\ell}(a \bullet \nu_{\Omega_\ell}(a)) \, d\mathcal{H}^n(a) = N_{\Omega_\ell}(\sigma\,\phi\, \varphi_n)
$$
and we pass to limit to obtain \eqref{lem: Minkowski formulaeII}.
\end{proof}


\section{Heintze-Karcher inequality for varifolds of bounded mean curvature}\label{section HK}

The next lemma provides fundamental structural properties of a set of finite perimeter whose reduced boundary is a varifold of bounded mean curvature. 

\begin{lemma}\label{lem varifolds}
	Let $ E \subseteq \mathbb{R}^{n+1} $ be a set of finite perimeter in $ \mathbb{R}^{n+1} $ and positive volume such that $ V_E $ is a varifold of bounded mean curvature. 
	
	Then there exists a closed set $ C \subseteq \mathbb{R}^{n+1} $ with non-empty interior such that 
	\begin{equation}\label{lem varifolds conclusion 0}
		\spt \| V_C \| = \partial C, \qquad \mathcal{L}^{n+1}((C \setminus E) \cup (E \setminus C)) =0,
	\end{equation}
	\begin{equation}\label{lem varifolds conclusion 1}
		\mathcal{H}^n\big[\partial C \setminus \big(\partial^\ast C \cap \pi_0(\nor(C))\big)\big]=0,
	\end{equation} 
	\begin{equation}\label{lem varifolds conclusion 2}
		\mathcal{H}^n\big(\nor(\partial C)\restrict Z\big) =0 \quad \textrm{whenever $ Z \subseteq \partial C $ with $ \mathcal{H}^n(Z) =0 $,}
	\end{equation} 
	\begin{equation}\label{lem varifolds conclusion 4a}
		\nor(C)\restrict \partial^\ast C = \{(x,\nu_C(x)): x \in \partial^\ast C \cap \pi_0(\nor(C))\},
	\end{equation}
	\begin{equation}\label{lem varifolds conclusion 4b}
		\mathcal{H}^n(\nor(C) \setminus \overline{\nu}_C(\partial^\ast C)) =0,
	\end{equation}
	and 
	\begin{equation}\label{lem varifolds conclusion 3}
		\sum_{i=1}^n \kappa_{C,i}(x,u) = \bm{h}(V_C, x) \bullet u \quad \textrm{for $ \mathcal{H}^n $ a.e.\ $(x,u)\in \nor(C) $.} 
	\end{equation}
\end{lemma}

\begin{proof}
	Since $ V_E $ is a varifold of bounded mean curvature, it follows from \cite[Theorem 8.6]{Allard72} that $ \Theta^n(\mathcal{H}^n \restrict \partial^\ast E, \cdot) $ is an upper-semicontinuous function on $ \mathbb{R}^{n+1} $. We conclude that $\Theta^n(\mathcal{H}^n \restrict \partial^\ast E, x) \geq 1 $ for every $ x \in \overline{\partial^\ast E} $ and $ \mathcal{H}^n( \overline{\partial^\ast E} \setminus \partial^\ast E) =0 $ by standard density results. We infer from \cite[Lemma 6.2]{DeRosaetall} that there exists a closed set $ C\subseteq \mathbb{R}^{n+1} $ with non-empty interior such that \eqref{lem varifolds conclusion 0} holds and
	\begin{equation}\label{lem varifolds eq2}
		\mathcal{H}^n(\partial C \setminus \partial^\ast C) =0.
	\end{equation}
For every  $ x \in \partial^\ast C $,	by Allard-Duggan regularity theorem \cite{Duggan86}, there exist an open set $ U \subseteq \mathbb{R}^{n+1} $ with $ x \in U $, an open set $ V \subseteq \mathbb{R}^n $ and  a function $ f \in W^{2,p}(V, \mathbb{R}) $ for every $ p < \infty $, such that $ U \cap  \partial^\ast C $ coincides with the graph of $ f $, up to a rotation in $\mathbb{R}^{n+1} $. We now recall Reshetnyak differentiability theorem, see \cite[Theorem 2]{Reshetnyakdifferentiability}: if $ f \in W^{2,p}(V, \mathbb{R}) $ with $ p > n $, then for $ \mathcal{L}^n $ a.e.\ $ a \in V $ there exists a polynomial function $ P_a $ of degree at most $ 2 $ such that $ P_a(a) = f(a) $ and 
	\begin{equation}\label{lem varifolds eq1}
		\lim_{b \to a}\frac{|f(b) -P_a(b)|}{|b-a|^2} =0. 
	\end{equation} 
	Denoting with $ G $ the graph of $ f $ and with $ \nu_G $ a continuous unit-normal on $ G $, we easily conclude from \eqref{lem varifolds eq1} that for $ \mathcal{H}^n $ a.e.\ $ x \in G $ there exists $ s > 0 $ such that
	$$\dist(x+s\nu_G(x), G) = s \quad \textrm{and} \quad \dist(x-s\nu_G(x), G). $$
	Henceforth, it follows that 
	$$ \mathcal{H}^n(\partial^\ast C \setminus \pi_0(\nor(C))) =0, $$
	and we obtain \eqref{lem varifolds conclusion 1} from \eqref{lem varifolds eq2}.
	
	The assertion in \eqref{lem varifolds conclusion 2} is proved in \cite[Theorem 3.8]{SantilliBulletin}\footnote{In \cite[Theorem 3.8]{SantilliBulletin} the desired conclusion is actually achieved in the more general class of $(m,h)$-sets.}.  To prove \eqref{lem varifolds conclusion 4a} we recall  from De Giorgi theorem \cite[Theorem 3.59]{AFP00} that if $ x \in \partial^\ast C $ then the sets $ \frac{C-x}{r} $ converge in measure as $ r \to 0 $ to an halfspace perpendicular to $ \nu_C(x) $. Consequently if $(x,u) \in \nor(C) $ and $ x \in \partial^\ast C $, we readily conclude $ u = \nu_C(x) $. This proves \eqref{lem varifolds conclusion 4a}, whence we infer that 
	\begin{equation}\label{lem varifolds eq5}
		\nor(C) \restrict(\partial C \setminus \partial^\ast C) = \nor(C) \setminus \overline{\nu}_C(\partial^\ast C).
	\end{equation}  
	Then \eqref{lem varifolds conclusion 4b} follows from \eqref{lem varifolds eq5}.
	
	Finally we prove \eqref{lem varifolds conclusion 3}. By the $C^2 $-rectifiability of Menne \cite{Menne13} we obtain a countable family $ F $ of embedded $ C^2 $-hypersurfaces  such that \begin{equation}\label{lem varifolds eq3}
		\mathcal{H}^n\Big(\partial C \setminus \bigcup F\Big) =0
	\end{equation} 
	and for each $ \Sigma \in F $ we have that 
	\begin{equation}\label{lem varifolds eq4}
		\bm{h}(\Sigma, x) = \bm{h}(V_C, x) \qquad \textrm{for $ \mathcal{H}^n $ a.e.\ $ x \in \Sigma \cap \partial C $.}
	\end{equation}  
	(Here $\bm{h}(\Sigma, x) $ is the mean curvature vector of $ \Sigma $.) For each $ \Sigma \in F $, if  $ R_\Sigma \subseteq \Sigma \cap \partial C $ is a set as in Lemma \ref{lem: Santilli20II} and $ R'_\Sigma  $ is the set of $ x \in \Sigma \cap \partial C $ such that \eqref{lem varifolds eq4} holds, we have that $ \mathcal{H}^n\big((\Sigma \cap \partial C) \setminus (R_\Sigma\cap R'_\Sigma)\big) =0 $ and it follows from \eqref{lem varifolds conclusion 2} that 
	$$ \mathcal{H}^n\big[\nor(C) \restrict\big((\Sigma \cap \partial C) \setminus (R_\Sigma\cap R'_\Sigma)\big)\big] = 0; $$
	therefore we can conclude that 
	$$ \sum_{i=1}^n\kappa_{C,i}(x,u) = \sum_{i=1}^n\kappa_{\Sigma,i}(x,u) = \bm{h}(\Sigma, x) \bullet u = \bm{h}(V_C,x)\bullet u $$
	for $ \mathcal{H}^n $ a.e.\ $(x,u) \in \nor(C) \restrict (\Sigma \cap \partial C) $ and for $ 1 \leq i \leq n $. Now \eqref{lem varifolds conclusion 3} follows from \eqref{lem varifolds eq3}.
\end{proof}

We can now state and prove the Heintze-Karcher inequality for sets of finite perimeter and bounded distributional mean curvature.

\begin{theorem}\label{HK}
Let $ E \subseteq \mathbb{R}^{n+1} $ be a set of finite perimeter and finite volume such that $ V_E $ is a varifold of bounded mean curvature and $$ \bm{h}(V_E,x)\bullet \nu_E(x)\geq 0 \quad \textrm{for $ \mathcal{H}^n $ a.e.\ $ x \in \partial^\ast E $.} $$ 

Then 
\begin{equation}\label{HKeq}
	\mathcal{L}^{n+1}(E) \leq \frac{n}{n+1} \int_{\partial^\ast E}\frac{1}{\bm{h}(V_E,x) \bullet \nu_E(x)}\, d\mathcal{H}^n(x),
\end{equation} 
and the equality is achieved if and only if $ E $ is $ \mathcal{L}^{n+1} $-almost equal to a finite union of closed ball with disjointed interiors.
\end{theorem}

\begin{proof}
We assume $ \mathcal{L}^{n+1}(E) > 0 $.	Define $ F = \mathbb{R}^{n+1} \setminus E $ and notice that $$ \partial^\ast F = \partial^\ast E \quad \textrm{and} \quad \nu_F = -\nu_E. $$ Since $ V_F = V_E $ is a varifold of bounded mean curvature, we can apply Lemma \ref{lem varifolds} to find a closed set $ K \subseteq \mathbb{R}^{n+1} $ with non empty interior such that \eqref{lem varifolds conclusion 0}-\eqref{lem varifolds conclusion 3}  hold with $ E $ and $ C $ replaced by $ F $ and $ K $, respectively. In particular, (since $ V_E = V_K $ and $ \nu_K = -\nu_E $) 
$$ \sum_{i=1}^n \kappa_{K,i}(x,u) = \bm{h}(V_E, x) \bullet u \quad \textrm{for $ \mathcal{H}^n $ a.e.\ $(x,u) \in \nor(K) $}, $$
$$ \nor(K) \restrict \partial^\ast K = \{(x, -\nu_E(x)) : x \in \partial^\ast K \cap \pi_0(\nor(K))\} $$
and $$ \mathcal{H}^n(\nor(K) \setminus (\nor(K) \restrict \partial^\ast K)) =0. $$
Henceforth,  $$\sum_{i=1}^n \kappa_{K,i}(x,u) = -\bm{h}(V_E, x) \bullet \nu_E(x) \leq 0 \quad \textrm{for $ \mathcal{H}^n $ a.e.\ $(x,u) \in \nor(K) $} $$
and it follows from Theorem \ref{HK general} and area formula that
\begin{flalign*}
	(n+1)\mathcal{L}^{n+1}(E) & \leq \int_{\nor(K)} J_n^{\nor(K)}\pi_0(x,u)\, \frac{n}{\bm{h}(V_E, x) \bullet \nu_E(x)}\, d\mathcal{H}^n(x,u)\\
	&  = \int_{\partial^\ast  E}\frac{n}{\bm{h}(V_E, x) \bullet \nu_E(x)}\, d\mathcal{H}^n(x).
\end{flalign*}  
\noindent If the equality is achieved in \eqref{HKeq}, since $ V_E $ is a varifold of bounded mean curvature, we obtain the conclusion directly from the rigidity statement in Theorem \ref{HK general}.
\end{proof}

\begin{corollary}\label{HK corollary}
Let $ E \subseteq \mathbb{R}^{n+1} $ be a set of finite perimeter with positive and finite volume such that $ V_E $ is a varifold of bounded mean curvature and
$$ \bm{h}(V_E,x) \bullet \nu_E(x) \geq \frac{n\mathcal{H}^n(\partial^\ast E)}{(n+1)\mathcal{L}^{n+1}(E)} \quad \textrm{for $ \mathcal{H}^n $ a.e.\ $x \in \partial^\ast E $.} $$

Then $ E $ is $ \mathcal{L}^{n+1} $ almost equal to a finite union of disjoint open balls of the same radius $ \rho = \frac{(n+1)\mathcal{L}^{n+1}(E)}{\mathcal{H}^n(\partial^\ast E)} $.
\end{corollary}

\begin{proof}
	The result can be deduced from Theorem \ref{HK} using the same argument as in  \cite[Corollary 5.16]{HugSantilli}. We leave the details to the reader. 
\end{proof}

\section{Proof of Theorem \ref{main}}\label{section main}

Let $ \Omega \subseteq \mathbb{R}^{n+1} $  be a set of finite perimeter such that the sequence $ \Omega_\ell $ converges in measure to $ \Omega $ and  $ \mathcal{H}^n(\partial \Omega_\ell) \to \mathcal{H}^n(\partial^\ast \Omega) $. We can assume that $ \mathcal{L}^{n+1}(\Omega) > 0 $, otherwise there is nothing to prove. It follows from \cite[Propositions 1.80 and 3.13]{AFP00}) that 
\begin{equation}\label{main strict convergence}
	\mathcal{H}^n \restrict \partial \Omega_\ell \to \mathcal{H}^n \restrict \partial^\ast \Omega \quad \textrm{and} \quad \nu_{\Omega_\ell}\mathcal{H}^n \restrict \partial \Omega_\ell \to \nu_\Omega \mathcal{H}^n \restrict \partial^\ast \Omega
\end{equation} 
weakly as Radon measures.  Employing Reshetnyak theorem \cite[Theorem 2.39]{AFP00} and recalling \eqref{first variation eq1} we conclude that
	$$ \delta V_{\Omega_\ell}(g) \to \delta V_\Omega(g) \quad \textrm{for all $ g \in C^\infty_c(\mathbb{R}^{n+1}, \mathbb{R}^{n+1}) $.  }  $$
We estimate
	\begin{flalign*}
		&| \delta V_{\Omega_\ell}(g)|   \leq  \int_{\partial \Omega_\ell} (H_{\Omega_\ell,1})^+|g|\, d\mathcal{H}^n + \int_{\partial \Omega_\ell} (H_{\Omega_\ell,1})^-|g|\, d\mathcal{H}^n \\
		&    \leq  M\,\int_{\{0 \leq H_{\Omega_\ell,1} \leq M\}}  |g|\, d\mathcal{H}^n   + \int_{\{H_{\Omega_\ell,1} > M\}}H_{\Omega_\ell,1}|g|\, d\mathcal{H}^n + \int_{\partial \Omega_\ell} (H_{\Omega_\ell,1})^-|g|\, d\mathcal{H}^n \\
		&  \leq   M\,\int_{\partial \Omega_\ell}  |g|\, d\mathcal{H}^n  + \bigg(\int_{\partial \Omega_\ell}|H_{\Omega_\ell, 1}|^n\, d\mathcal{H}^n\bigg)^{\frac{1}{n}}\bigg(\int_{\{H_{\Omega_\ell,1} > M\}}|g|^{\frac{n}{n-1}} \, d\mathcal{H}^n\bigg)^{\frac{n-1}{n}} \\
		& \quad \qquad \qquad \qquad \qquad  \qquad  \qquad  \qquad  \qquad  \qquad \qquad  \qquad + \int_{\partial \Omega_\ell} (H_{\Omega_\ell,1})^-|g|\, d\mathcal{H}^n. 
	\end{flalign*}
Since the sequence has finite total curvature we have that $$ \sup_{\ell \geq 1} \int_{\partial \Omega_\ell}|H_{\Omega_\ell, 1}|^n\, d\mathcal{H}^n < \infty, $$ and  the second integral on the right-side converges to zero as $ \ell \to \infty $ thanks to \eqref{main: hp3}. Noting that also the third integral on the right-side converges to zero since the sequence is asymptotically $(k-1)$-mean convex, we infer from \eqref{main strict convergence} that
	$$ | \delta V_\Omega(g) | \leq M \int_{\partial^\ast \Omega}  |g|\, d\mathcal{H}^n \quad \textrm{for $ g \in C^\infty_c(\mathbb{R}^{n+1}, \mathbb{R}^{n+1}) $.} $$
Henceforth, $ V_\Omega $ is a varifold of bounded mean curvature. Let $ C \subseteq \mathbb{R}^{n+1} $ be the closed set given by Lemma \ref{lem varifolds} such that \mbox{$ \mathcal{L}^{n+1}\big(C \triangle \Omega \big) =0 $.} In applying the conclusion \eqref{lem varifolds conclusion 2} of Lemma \ref{lem varifolds} it is useful to notice the trivial inclusion $ \nor(C) \subseteq \nor(\partial C) $.  Additionally, it follows from \eqref{mass normal bundle} that $ \sup_{\ell \geq 1}\bm{M}(N_{\Omega_\ell}) < \infty $ and we deduce from Federer-Fleming compactness theorem, cf.\ \cite{Simonbook} that, up to subsequences, $$ N_{\Omega_\ell} \to T \quad \textrm{weakly in the sense of currents}, $$ with $ T $ being an $ n$-dimensional integer-rectifiable current compactly supported in $ \mathbb{R}^{n+1} \times \mathbb{S}^n $. Since $ \partial N_{\Omega_\ell} =0 $ and $ N_{\Omega_\ell} \restrict \alpha_{\mathbb{R}^{n+1}} =0 $, we readily infer that $ T $ is a Legendrian cycle of $ \mathbb{R}^{n+1} $.
   
We have just seen that two different kind of limits are at play here: the varifolds $ V_{\Omega_\ell} $ converging to $ V_C $ and the normal cycles $ N_{\Omega_\ell} $ converging as currents to $ T $. To prove the theorem we need to understand how these two limits are related. We start proving the following two statements:
	\begin{equation}\label{main conclusion1}
	\mathcal{H}^{n}\big(\nor(C) \setminus W_T^{(n)}\big)=0
\end{equation} 
\begin{equation}\label{main conclusion2}
H_{T,1}(x, u) = \bm{h}(V_C, x) \bullet u \qquad \textrm{for $ \mathcal{H}^n $ a.e.\ $ (x,u) \in \nor(C)  $.}
\end{equation} 
By Reshetnyak continuity theorem \cite[Theorem 2.39]{AFP00}, we have that
	\begin{equation*}
		(N_{\Omega_\ell}\restrict \varphi_n)(\phi) = \int_{\partial \Omega_\ell}\phi(a, \nu_{\Omega_\ell}(a))\, d\mathcal{H}^n(a) \rightarrow \int_{\partial^\ast C}\phi(a, \nu_C(a))\, d\mathcal{H}^n(a)
	\end{equation*} 
	for every $ \phi \in \mathcal{D}^0(\mathbb{R}^{n+1}\times \mathbb{R}^{n+1})  $. Since  $N_{\Omega_\ell}\restrict \varphi_n \to T \restrict \varphi_n $ weakly, it follows from Lemma \ref{lem: representation curvature measures} and Remark \ref{remark: jacodian of pi0} that
	\begin{flalign}\label{lem 3 eq1}
		\int_{\partial^\ast C}\phi(a, \nu_C(a))\, d\mathcal{H}^n(a) & = (T \restrict \varphi_n)(\phi) \notag \\
		& = \int_{W_T^{(n)}} J_n^{W_T}\pi_0(a,u)i_T(a,u)\phi(a,u)\, d\mathcal{H}^n(a,u)
	\end{flalign}
	for $ \phi \in \mathcal{D}^0(\mathbb{R}^{n+1}\times \mathbb{R}^{n+1}) $. Let $ R = \partial^\ast C \cap \pi_0(\nor(C)) $ and choose a $ \mathcal{H}^n $-measurable set $ Q \subseteq \nor(C) $ with finite $ \mathcal{H}^n $-measure.   We notice from \eqref{lem varifolds conclusion 1} and \eqref{lem varifolds conclusion 2} of Lemma \ref{lem varifolds}  that $\mathcal{H}^n(Q \setminus (Q \restrict R)) =0 $. The latter in combination with area formula, \eqref{lem varifolds conclusion 4a} of Lemma \ref{lem varifolds} and \eqref{lem 3 eq1} allows to conclude
	\begin{flalign*}
		& \int_{Q}J_n^{Q} \pi_0(x,u)\, \phi(x,u) \, d\mathcal{H}^n(x,u) \notag \\
		& \qquad = \int_{Q\restrict R}J_n^{Q} \pi_0(x,u)\, \phi(x,u) \, d\mathcal{H}^n(x,u) \notag\\
		& \qquad = \int_{\pi_0(Q \restrict R)} \phi(x,\nu_C(x)) \, d\mathcal{H}^n(x) \notag\\
		& \qquad \leq    \int_{W_T^{(n)}} J_n^{W_T}\pi_0(x,u)i_T(x,u)\phi(x,u)\, d\mathcal{H}^n(x,u) 
	\end{flalign*} 
	for every $ \phi \in \mathcal{D}^0(\mathbb{R}^{n+1}\times \mathbb{R}^{n+1}) $ with $ \phi \geq 0 $.  It follows that 
	\begin{flalign}\label{lem 3 eq 3}
		& \int_{Q \cap S}J_n^{Q} \pi_0(x,u) \, d\mathcal{H}^n(x,u) \notag \\
		& \qquad  \qquad  \leq \int_{W_T^{(n)}\cap S} J_n^{W_T}\pi_0(x,u)i_T(x,u)\, d\mathcal{H}^n(x,u) 
	\end{flalign} 
	for every $ \mathcal{H}^n $-measurable set $ S \subseteq \mathbb{R}^{n+1} \times \mathbb{R}^{n+1} $. We prove now that
	\begin{equation}\label{main positive jacobian}
		J_n^Q\pi_0(x,u) > 0 \quad \textrm{for $ \mathcal{H}^n $ a.e.\ $(x,u)\in Q $.}
	\end{equation}
	If $ P = \{(x,u)\in Q: J_n^Q\pi_0(x,u)> 0\} $, then 
	$$ 0 = \int_{Q \setminus P}J_n^Q\pi_0(x,u)\, d\mathcal{H}^n(x,u) = \int_{\pi_0(Q \setminus P)}\mathcal{H}^0(\pi_0^{-1}(x) \cap (Q \setminus P))\, d\mathcal{H}^n(x), $$
	whence we infer, since $\mathcal{H}^0(\pi_0^{-1}(x) \cap (Q \setminus P))\geq 1 $ for $ \mathcal{H}^n $ a.e.\ $x \in \pi_0(Q \setminus P)$, that 
	$$ \mathcal{H}^n(\pi_0(Q \setminus P)) =0. $$
	It follows from \eqref{lem varifolds conclusion 2} of Lemma \ref{lem varifolds} that $ \mathcal{H}^n(Q \setminus P) =0 $ and \eqref{main positive jacobian} is proved. Thanks to \eqref{main positive jacobian},  we can  choose $ S = (\mathbb{R}^{n+1} \times \mathbb{R}^{n+1})\setminus W_T^{(n)} $ in \eqref{lem 3 eq 3} and deduce  that $$ \mathcal{H}^n(Q \setminus W_T^{(n)}) =0. $$ Since the latter holds for every $ Q \subseteq \nor(C)  $ with finite $ \mathcal{H}^n $-measure, we obtain \eqref{main conclusion1}.
Notice that \eqref{main conclusion1} implies that $ \mathcal{H}^n(\nor(C))< \infty $. It follows that
$$ \Tan^n(\mathcal{H}^n \restrict \nor(C), (x,u)) = \Tan^n(\mathcal{H}^n \restrict W_T, (x,u)) $$ for $ \mathcal{H}^n $ a.e.\ $ (x,u) \in \nor(C) $ and, combining Theorem \ref{theo: Rataj-Zaehle} with Lemma \ref{lem: Santilli20} we conclude that
\begin{equation}\label{lem 3 eq 6}
\kappa_{C, i}(x,u) = \kappa_{T,i}(x,u)  \quad \textrm{for $ \mathcal{H}^n $ a.e.\ $(x,u) \in \nor(C) $}
\end{equation} 
for $ 1 \leq i \leq n $. Now \eqref{main conclusion2} follows from Lemma \ref{lem varifolds}.

We notice that it follows from  \eqref{normal cycle C2 boundaries eq2} that
\begin{equation*}
(N_{\Omega_\ell} \restrict \varphi_{n-k})(\phi) 
 = \lambda (N_{\Omega_\ell}\restrict \varphi_n)(\phi) + \int_{\partial \Omega_\ell} (H_{\Omega_\ell, k}(x)-\lambda)\, \phi(x,\nu_{\Omega_\ell}(x))\, d\mathcal{H}^n(x)
\end{equation*}
for $ \phi \in \mathcal{D}^0(\mathbb{R}^{n+1}\times \mathbb{R}^{n+1}) $. Since the second integral on the right side goes to $ 0 $ as $ \ell \to \infty $, we obtain 
\begin{equation*}\label{lem 2 eq 1}
	T \restrict \varphi_{n-k} = \lambda (T \restrict \varphi_n).
\end{equation*} 
Henceforth it follows from Lemma \ref{lem: representation curvature measures} 
\begin{flalign*}
&	\int_{W_T^{(n)}} \phi(x,u)\, i_T(x,u)\, \mathcal{J}_T(x,u)\, (H_{T, k}(x,u) -\lambda)\, d\mathcal{H}^n(x,u)  \\
& \qquad = \int_{W_T \setminus W_T^{(n)}} \phi(x,u)\, i_T(x,u)\, \mathcal{J}_T(x,u)\, H_{T,k}(x,u)\,d\mathcal{H}^n(x,u)
\end{flalign*}
for every $ \phi \in \mathcal{D}^0(\mathbb{R}^{n+1}\times \mathbb{R}^{n+1}) $. Since $ i_T(x,u) \mathcal{J}_T(x,u)>0 $ for $ \mathcal{H}^n$ a.e.\ $(x,u)\in W_T $, we infer 
\begin{equation}\label{main H_k 1}
	 H_{T, k}(x,u) =0 \quad \textrm{for $ \mathcal{H}^n $ a.e.\ $(x,u)\in W_T \setminus W_T^{(n)} $}
\end{equation} 
and 
\begin{equation}\label{main H_k 2}
	H_{T, k}(x,u)= \lambda \quad \textrm{for $ \mathcal{H}^n $ a.e.\ $(x,u)\in  W_T^{(n)} $.}
\end{equation}
Employing Lemma \ref{lem: Minkowski formulae} we obtain
\begin{flalign}\label{main: minkowskiI}
&	(n-k+1)\int i_T(x,u)\,\mathcal{J}_T(x,u)\, H_{T,k-1}(x,u)\,d\mathcal{H}^n(x,u) \\
& \qquad = k \int (x \bullet u)	\, i_T(x,u)\,\mathcal{J}_T(x,u)\, H_{T,k}(x,u)\,d\mathcal{H}^n(x,u) \notag \\
& \qquad = k  \lambda \int_{W_T^{(n)}} (x \bullet u)	\, i_T(x,u)\,\mathcal{J}_T(x,u)\, d\mathcal{H}^n(x,u) \notag \\
& \qquad = k \lambda (n+1)\, \mathcal{L}^{n+1}(C)\notag 
\end{flalign}
and we conclude that (notice that $ 0 <\mathcal{L}^{n+1}(C) < \infty $)
\begin{equation}\label{main: lambda}
	H_{T,k}(x,u) = \lambda = \frac{(n-k+1)T(\varphi_{n-k+1})}{k (n+1)\, \mathcal{L}^{n+1}(C)}
\end{equation}
for $ \mathcal{H}^n $ a.e.\ $(x,u)\in  W_T^{(n)} $. Since, by Lemma \ref{lem asymptotically mean convexity}, $ H_{T,k-1}(x,u)\geq 0 $ for $ \mathcal{H}^n $ a.e.\ $(x,u)\in W_T $, we conclude from \eqref{main: minkowskiI} that $ \lambda \geq 0 $.

We claim now that $ \lambda > 0 $ and we prove it by contradiction. Henceforth we assume $ \lambda =0 $ and we prove by induction that 
\begin{equation}\label{main: lambda =0}
	H_{T,j}(x,u) =0 \quad \textrm{for $ \mathcal{H}^n $ a.e.\ $ (x,u)\in W_T $}
\end{equation}  
for $ j =0, \ldots , k $. If $ j = k $ then \eqref{main: lambda =0} follows from \eqref{main H_k 1} and  \eqref{main H_k 2}.  If $ i \in \{1, \ldots , k\} $ and $ H_{T,k-i+1}(x,u) =0 $ for $ \mathcal{H}^n $ a.e.\ $(x,u)\in W_T $, we infer from Minkowski formula in Lemma \ref{lem: Minkowski formulae} that 
\begin{flalign*}
&0 =(k-i+1)\int (x\bullet u)i_T(x,u)\, \mathcal{J}_T(x,u)\, H_{T, k-i+1}(x,u)\, d\mathcal{H}^n(x,u)\\
& \qquad  = (n-k+i)\int i_T(x,u)\, \mathcal{J}_T(x,u)\, H_{T, k-i}(x,u)\, d\mathcal{H}^n(x,u).
\end{flalign*}  
Since $ H_{T, k-i}(x,u)\geq 0 $   (by Lemma \ref{lem asymptotically mean convexity}) and $ i_T(x,u)\mathcal{J}_T(x,u) > 0 $ for $ \mathcal{H}^n $ a.e.\ $(x,u)\in W_T $, we conclude that 
$$ H_{T,k-i}(x,u) =0 \quad \textrm{for $ \mathcal{H}^n $ a.e.\ $(x,u)\in W_T $.} $$
Henceforth, \eqref{main: lambda =0} is proved for all $ j \in \{0, \ldots , k\} $. It follows that  $ H_{T,0}(x,u) =0 $ for $ \mathcal{H}^n $ a.e.\ $(x,u)\in W_T $ and $ \mathcal{L}^{n+1}(\Omega) =0 $ by formula \eqref{lem: Minkowski formulaeII} of Lemma \ref{lem: Minkowski formulae}. This a contradiction, and consequently we conclude  that $ \lambda $ is positive.

Since $ H_{T,i}(x,u) \geq 0 $ for $ \mathcal{H}^n $ a.e.\ $(x,u)\in W_ T $ and for every $ 1 \leq i \leq k $, we can apply \cite[Lemma 2.2]{HugSantilli} to obtain 
\begin{equation}\label{main: lower bound}
	\frac{H_{T,1}(x,u)}{{n \choose 1}} \geq \Bigg(\frac{H_{T,2}(x,u)}{{n \choose 2}}\Bigg)^{\frac{1}{2}} \geq \ldots \geq \Bigg(\frac{H_{T,k}(x,u)}{{n \choose k}}\Bigg)^{\frac{1}{k}} \geq \Bigg( \frac{\lambda}{{n \choose k}}\Bigg)^{\frac{1}{k}}
\end{equation} 
for $ \mathcal{H}^n $ a.e.\ $(x,u)\in W_T^{(n)} $. Noting by \eqref{main conclusion1} and Remark \ref{remark: jacodian of pi0} that $$ \mathcal{J}_T(x,u) = J_n^{W_T}\pi_0(x,u) = J_n^{\nor(C)}\pi_0(x,u) \quad \textrm{for $\mathcal{H}^n $ a.e.\ $(x,u)\in \nor(C) $} $$ and $ H_{T,k-1}(x,u) \geq 0 $ for $ \mathcal{H}^n $ a.e.\ $(x,u)\in W_T $ by Lemma \ref{lem asymptotically mean convexity}, we employ  Lemma \ref{lem: Minkowski formulae}, \eqref{main: lower bound}, \eqref{main conclusion1} and area formula to conclude
\begin{flalign}\label{main MinkowskiII}
& (n-k+1)\int \mathcal{J}_T(x,u)\, i_T(x,u)\, H_{T,k-1}(x,u)\, d\mathcal{H}^n(x,u)\\
& \qquad  \geq (n-k+1)\int_{W_T^{(n)}} \mathcal{J}_T(x,u)\, i_T(x,u)\, H_{T,k-1}(x,u)\, d\mathcal{H}^n(x,u)\notag \\
& \qquad \geq  (n-k+1)\,\lambda^{\frac{k-1}{k}}\,{n \choose k}^{\frac{1-k}{k}}{n \choose k-1}\, \int_{W_T^{(n)}} \mathcal{J}_T(x,u)\, i_T(x,u)\, d\mathcal{H}^n(x,u)\notag \\
& \qquad  \geq  (n-k+1)\,\lambda^{\frac{k-1}{k}}\,{n \choose k}^{\frac{1-k}{k}}{n \choose k-1}\, \int_{\nor(C)} J_n^{\nor(C)}\pi_0(x,u)\, i_T(x,u)\, d\mathcal{H}^n(x,u)\notag \\
& \qquad \geq (n-k+1)\,\lambda^{\frac{k-1}{k}}\,{n \choose k}^{\frac{1-k}{k}}{n \choose k-1}\,\mathcal{H}^n(\partial^\ast C). \notag 
\end{flalign}
Combining \eqref{main: minkowskiI} and \eqref{main MinkowskiII} we conclude
$$\Bigg( \frac{\lambda}{{n \choose k}}\Bigg)^{\frac{1}{k}} \geq \frac{\mathcal{H}^n(\partial^\ast C)}{(n+1)\mathcal{L}^{n+1}(C)}, $$
whence we infer in combination with \eqref{main: lower bound} and \eqref{main conclusion2} that 
\begin{equation}
\bm{h}(V_C,x) \bullet u \geq \frac{n\mathcal{H}^n(\partial^\ast C)}{(n+1)\mathcal{L}^{n+1}(C)} \quad \textrm{for $ \mathcal{H}^n $ a.e.\ $(x,u)\in \nor(C) $.}
\end{equation}
We conclude from Corollary \ref{HK corollary} that $ C $ is the union of finitely many closed balls of the same radius $  \rho = \frac{(n+1)\mathcal{L}^{n+1}(E)}{\mathcal{H}^n(\partial^\ast E)}  $ with disjointed interiors. 

We now use Lemma \ref{lem: Santilli20II} with $ \Sigma $ replaced by a sphere of radius $ \rho $: noting that $ \nor(C) \subseteq \nor(\partial C) $, this lemma allows to conclude in combination with \eqref{lem varifolds conclusion 2} of Lemma \ref{lem varifolds},  that $$ \kappa_{C,1}(x,u) = \ldots = \kappa_{C,n}(x,u) = \frac{1}{\rho} \quad \textrm{for $ \mathcal{H}^n $ a.e.\ $(x,u)\in \nor(C) $}. $$ 
Noting that $ \kappa_{T,i}(x,u) = \kappa_{C,i}(x,u) $ for $ \mathcal{H}^n $ a.e.\ $(x,u)\in  \nor(C) $ and for $ i = 1, \ldots, n $, as proved in \eqref{lem 3 eq 6}, we conclude from \eqref{main H_k 2} that
$$ \lambda = H_{T,k}(x,u) = H_{C,k}(x,u) = {n \choose k}\rho^{-k} $$
for $ \mathcal{H}^n $ a.e.\ $(x,u)\in \nor(C) $. Since $ \mathcal{H}^n(\nor(C)) > 0 $ we finally conclude that
$\lambda = {n \choose k}\rho^{-k} $.

\appendix

\section{Soap bubbles with positive reach}\label{section appendix}

In this appendix we first generalize \cite[Theorem A]{HugSantilli}, removing the restriction  $ \lambda \in \mathbb{R} \setminus \{0\} $. Notice that in  \cite[Theorem A]{HugSantilli} a stronger hypothesis is at play than in \cite[Theorem 6.15]{HugSantilli} and  I do not know if the restriction $ \lambda \in \mathbb{R} \setminus \{0\} $ can also be removed in \cite[Theorem 6.15]{HugSantilli}. We use the notation and the terminology of \cite{HugSantilli}, for which we refer to \cite[Section 2.2]{HugSantilli}.

\begin{theorem}\label{soap bubble th positive reach}
Let $ \phi $ be a uniformly convex $ C^2 $-norm, $ k \in \{1, \ldots , n\} $ and let $ C \subseteq \mathbb{R}^{n+1} $ be a set of positive reach with positive and finite volume. Assume that 
\begin{equation}\label{soap bubble th positive reach: 1}
\textrm{$\Theta^\phi_{n-i}(C, \cdot) $ is a non-negative measure for $ i =1, \ldots , k-1 $}
\end{equation}
and 
\begin{equation}\label{soap bubble th positive reach: 2}
\Theta^\phi_{n-k}(C, \cdot) = \lambda \Theta^\phi_n(C, \cdot) \quad \textrm{for some $ \lambda \in \mathbb{R} $.}
\end{equation}
Then $ C $ is a finite union of disjoint union of rescaled and translated Wulff shapes of radius $ \rho = \frac{(n+1)\mathcal{L}^{n+1}(C)}{\mathcal{P}^\phi(C)} $ and $ \lambda = {n \choose r} \frac{\rho}{r+1} $.
\end{theorem}

\begin{proof}
The main point is to prove that $ \lambda > 0 $, then the conclusion can be deduced from \cite[Theorem A]{HugSantilli}. The proof of the positivity of $ \lambda $ follows closely the idea used in the proof of Theorem \ref{main}.

We notice from \cite[Lemma 6.14]{HugSantilli} that
	\begin{equation}\label{soap bubble th positive reach: eq1}
\bm{H}^\phi_{C,r}(a, \eta) =0 \quad \textrm{for $ \mathcal{H}^n $ a.e.\ $(a, \eta)\in N^\phi(C) \setminus \widetilde{N}^\phi_n(C) $} 
	\end{equation}
and
\begin{equation}\label{soap bubble th positive reach: eq2}
	\bm{H}^\phi_{C,r}(a, \eta) = (r+1)\lambda \quad \textrm{for $ \mathcal{H}^n $ a.e.\ $(a, \eta)\in  \widetilde{N}^\phi_n(C) $.}
\end{equation}
Assume by contradiction that $ \lambda =0 $. Now we prove by induction that 
\begin{equation}\label{soap bubble th positive reach: lambda}
	\bm{H}^\phi_{C,j}(a, \eta) =0 \quad \textrm{for $ \mathcal{H}^n $ a.e.\ $ (x,u)\in N^\phi(C) $}
\end{equation}  
for $ j =0, \ldots , r $. If $ j = r $ then \eqref{soap bubble th positive reach: lambda} follows from \eqref{soap bubble th positive reach: eq1} and  \eqref{soap bubble th positive reach: eq2}.  If $ i \in \{1, \ldots , r\} $ and $ \bm{H}^\phi_{C,r-i+1}(a, \eta) =0$ for $ \mathcal{H}^n $ a.e.\ $(a, \eta)\in N^\phi(C) $, we infer from Minkowski formula in \cite[Theorem 6.8]{HugSantilli}
\begin{flalign*}
	&0 =(r-i+1)\int_{N^\phi(C)} (a\bullet \bm{n}^\phi(\eta))\, J^\phi_C(a, \eta) \, \bm{H}^\phi_{C,r-i+1}(a, \eta)\, d\mathcal{H}^n(a, \eta)\\
	& \qquad  = (n-r+i)\int_{N^\phi(C)}\phi(\bm{n}^\phi(\eta))\, J^\phi_C(a,\eta)\,\bm{H}^\phi_{C,r-i}(a, \eta)\, d\mathcal{H}^n(a,\eta).
\end{flalign*}  
Since it follows from \eqref{soap bubble th positive reach: 1} and \cite[Definition 6.2]{HugSantilli} that $$ \bm{H}^\phi_{C,r-i}(a, \eta) \geq 0 \quad \textrm{and} \quad  \phi(\bm{n}^\phi(\eta)) J^\phi_C(a, \eta) > 0 \quad \textrm{for $ \mathcal{H}^n $ a.e.\ $(a,\eta)\in N^\phi(C) $,}$$ we conclude that 
$$ \bm{H}^\phi_{C,r-i}(a, \eta) =0 \quad \textrm{for $ \mathcal{H}^n $ a.e.\ $(x,u)\in N^\phi(C) $.} $$
Henceforth, \eqref{soap bubble th positive reach: lambda} is proved for all $ j \in \{0, \ldots , r\} $. It follows that  $ \bm{H}^\phi_{C,0}(a, \eta) =0 $ for $ \mathcal{H}^n $ a.e.\ $(a, \eta)\in N^\phi(C) $ and $ \mathcal{L}^{n+1}(C) =0 $ by the second formula in \cite[Theorem 6.8]{HugSantilli}. This a contradiction, and consequently we conclude  that $ \lambda $ is positive. 

\end{proof}

Combining Theorem \ref{soap bubble th positive reach} with some standard facts on sets of positive reach we can prove the following uniqueness results for $ C^2 $-boundaries with $ k $-th mean curvature functions converging in $ L^1 $ to a constant. A completely analogous statement can be obtained in the anisotropic geometry induced by a uniformly convex $ C^2 $-norm. We choose to state Theorem \ref{positive reach} in the Euclidean setting to facilitate the comparison with Theorem \ref{main}.

\begin{theorem}\label{positive reach}
	Let $ k \in \{1, \ldots , n\} $, let $ \Omega_\ell $ be a compactly supported  and asymptotically $(k-1)$-mean convex sequence and let $ C_\ell $ be the closure of $ \Omega_\ell $. Suppose there exists $ \lambda \in \mathbb{R}$ such that 
	\begin{equation}\label{positive reach: hp}
		\lim_{\ell \to \infty}\int_{\partial \Omega_\ell}| H_{\Omega_\ell, k} -\lambda | \, d\mathcal{H}^n =0
	\end{equation}
	and there exists $\epsilon > 0 $ so that $ \reach(C_\ell)\geq \epsilon $ for all $ \ell \geq 1 $.
	
	If $ C \subseteq \mathbb{R}^{n+1} $  is an accumulation point with respect to Hausdorff convergence of $C_\ell  $ with positive volume, then $ C $ is a closed ball.
\end{theorem}

\begin{proof}
It follows from \cite[Lemma 6.12]{HugSantilli} that $ \reach(C) \geq \epsilon $. Moreover, the $ j $-th curvature measures $ \Theta_j(C_\ell, \cdot) $ associated with $ C_\ell $ converge weakly to $ \Theta_j(C, \cdot) $ for every $ j \in \{0, \ldots ,n\} $. Recalling that $ H_{\Omega_\ell, m} $ is the $ m $-th mean curvature function of $ \partial \Omega_\ell $ with respect to the outward-pointing normal  direction, and $ H_{\Omega_\ell, 0} \equiv 1 $, we notice from \cite[Lemma 6.4]{HugSantilli} that
\begin{flalign*}
(n-j+1)\Theta_j(C_\ell, B) = \int_{\partial \Omega_\ell} \bm{1}_B(a, \nu_{\Omega_\ell}(a))\, H_{\Omega_\ell, n-j}(a)\, d\mathcal{H}^n(a, \eta)
\end{flalign*}  
for every Borel set $ B \subseteq \mathbb{R}^{n+1}\times \mathbb{R}^{n+1} $ and $ j = 0, \ldots , n $. We consider the negative and positive part of $ \Theta_j(C_\ell, \cdot) $
$$ \Theta^{\pm}_j(C_\ell, \cdot) = \int_{\partial \Omega_\ell} \bm{1}_{(\cdot)}(a, \nu_{\Omega_\ell}(a))\, H_{\Omega_\ell, n-j}(a)^{\pm}\, d\mathcal{H}^n(a, \eta) \quad \textrm{for $ j = 1, \ldots, n $}, $$
which are non-negative Radon measures.  Noting that
$ \Theta^-_{n-i}(C_\ell, \cdot) $ weakly converge to $ 0 $ for $ i \in \{1, \ldots , k-1\} $, we conclude that $ \Theta^+_{n-i}(C_\ell, \cdot) = \Theta_{n-i}(C_\ell, \cdot) + \Theta^-_{n-i}(C_\ell, \cdot) $ weakly converge to $ \Theta_{n-i}(C, \cdot) $. Henceforth,  $ \Theta_{n-i}(C, \cdot)  $ is a non-negative Radon measure for $ i \in \{1, \ldots , k-1\} $.
Moreover, noting that 
$$ (k+1)\Theta_{n-k}(C_\ell, B) = \int_{\partial \Omega_\ell} \bm{1}_B(a, \nu_{\Omega_\ell}(a))\, (H_{\Omega_\ell,k} - \lambda)\, d\mathcal{H}^n(a) + \lambda \Theta_n(C_\ell, B) $$
for every Borel set $ B \subseteq \mathbb{R}^{n+1}\times \mathbb{R}^{n+1} $, we conclude from \eqref{positive reach: hp} that $ \Theta_{n-k}(C, \cdot) = \lambda \Theta_n(C, \cdot) $. Now the conclusion follows from Theorem \ref{soap bubble th positive reach}.
\end{proof}


\begin{thebibliography}{DMMN18}
	
	\bibitem[AFP00]{AFP00}
	Luigi Ambrosio, Nicola Fusco, and Diego Pallara.
	\newblock {\em Functions of bounded variation and free discontinuity problems}.
	\newblock Oxford Mathematical Monographs. The Clarendon Press, Oxford
	University Press, New York, 2000.
	
	\bibitem[Ale58]{Aleksandrov}
	A.~D. Aleksandrov.
	\newblock Uniqueness theorems for surfaces in the large. {V}.
	\newblock {\em Vestnik Leningrad. Univ.}, (no. 19):5--8, 1958.
	
	\bibitem[All72]{Allard72}
	William~K. Allard.
	\newblock On the first variation of a varifold.
	\newblock {\em Ann. of Math. (2)}, 95:417--491, 1972.
	
	\bibitem[CM17]{CiraoloMaggi}
	Giulio Ciraolo and Francesco Maggi.
	\newblock On the shape of compact hypersurfaces with almost-constant mean
	curvature.
	\newblock {\em Comm. Pure Appl. Math.}, 70(4):665--716, 2017.
	
	\bibitem[CNS85]{CaffarelliNirenberSpruck85}
	L.~Caffarelli, L.~Nirenberg, and J.~Spruck.
	\newblock The {D}irichlet problem for nonlinear second-order elliptic
	equations. {III}. {F}unctions of the eigenvalues of the {H}essian.
	\newblock {\em Acta Math.}, 155(3-4):261--301, 1985.
	
	\bibitem[CRV21]{CiraoloRoncoroniVezzoni}
	Giulio Ciraolo, Alberto Roncoroni, and Luigi Vezzoni.
	\newblock Quantitative stability for hypersurfaces with almost constant
	curvature in space forms.
	\newblock {\em Ann. Mat. Pura Appl. (4)}, 200(5):2043--2083, 2021.
	
	\bibitem[CV18]{CiraoloVezzoni}
	Giulio Ciraolo and Luigi Vezzoni.
	\newblock A sharp quantitative version of {A}lexandrov's theorem via the method
	of moving planes.
	\newblock {\em J. Eur. Math. Soc. (JEMS)}, 20(2):261--299, 2018.
	
	\bibitem[CW13]{ChangWang2014}
	Sun-Yung~Alice Chang and Yi~Wang.
	\newblock Inequalities for quermassintegrals on {$k$}-convex domains.
	\newblock {\em Adv. Math.}, 248:335--377, 2013.
	
	\bibitem[DM19]{MaggiDelgadino}
	Matias~Gonzalo Delgadino and Francesco Maggi.
	\newblock Alexandrov's theorem revisited.
	\newblock {\em Anal. PDE}, 12(6):1613--1642, 2019.
	
	\bibitem[DMMN18]{MaggiArma2018}
	Matias~G. Delgadino, Francesco Maggi, Cornelia Mihaila, and Robin Neumayer.
	\newblock Bubbling with {$L^2$}-almost constant mean curvature and an
	{A}lexandrov-type theorem for crystals.
	\newblock {\em Arch. Ration. Mech. Anal.}, 230(3):1131--1177, 2018.
	
	\bibitem[DRKS20]{DeRosaetall}
	Antonio De~Rosa, Slawomir Kolasinski, and Mario Santilli.
	\newblock Uniqueness of critical points of the anisotropic isoperimetric
	problem for finite perimeter sets.
	\newblock {\em Arch. Ration. Mech. Anal.}, 238(3):1157--1198, 2020.
	
	\bibitem[Dug86]{Duggan86}
	J.~P. Duggan.
	\newblock {$W^{2,p}$} regularity for varifolds with mean curvature.
	\newblock {\em Comm. Partial Differential Equations}, 11(9):903--926, 1986.
	
	\bibitem[Fed69]{Fed69}
	Herbert Federer.
	\newblock {\em Geometric measure theory}.
	\newblock Die Grundlehren der mathematischen Wissenschaften, Band 153.
	Springer-Verlag New York, Inc., New York, 1969.
	
	\bibitem[Fu98]{Fu98}
	Joseph H.~G. Fu.
	\newblock Some remarks on {L}egendrian rectifiable currents.
	\newblock {\em Manuscripta Math.}, 97(2):175--187, 1998.
	
	\bibitem[GL09]{GuanLi2009}
	Pengfei Guan and Junfang Li.
	\newblock The quermassintegral inequalities for {$k$}-convex starshaped
	domains.
	\newblock {\em Adv. Math.}, 221(5):1725--1732, 2009.
	
	\bibitem[GLL12]{GuanLiLi}
	Pengfei Guan, Junfang Li, and Yanyan Li.
	\newblock Hypersurfaces of prescribed curvature measure.
	\newblock {\em Duke Math. J.}, 161(10):1927--1942, 2012.
	
	\bibitem[HK78]{HeintzeKarcher}
	Ernst Heintze and Hermann Karcher.
	\newblock A general comparison theorem with applications to volume estimates
	for submanifolds.
	\newblock {\em Ann. Sci. \'{E}cole Norm. Sup. (4)}, 11(4):451--470, 1978.
	
	\bibitem[HS22]{HugSantilli}
	Daniel Hug and Mario Santilli.
	\newblock Curvature measures and soap bubbles beyond convexity.
	\newblock {\em Adv. Math.}, 411(part A):Paper No. 108802, 89, 2022.
	
	\bibitem[Hsi54]{Hsiung}
	Chuan-Chih Hsiung.
	\newblock Some integral formulas for closed hypersurfaces.
	\newblock {\em Math. Scand.}, 2:286--294, 1954.
	
	\bibitem[JN23]{JulinNinikoski}
	Vesa Julin and Joonas Niinikoski.
	\newblock Quantitative {A}lexandrov theorem and asymptotic behavior of the
	volume preserving mean curvature flow.
	\newblock {\em Anal. PDE}, 16(3):679--710, 2023.
	
	\bibitem[Mag12]{Maggibook}
	Francesco Maggi.
	\newblock {\em Sets of finite perimeter and geometric variational problems},
	volume 135 of {\em Cambridge Studies in Advanced Mathematics}.
	\newblock Cambridge University Press, Cambridge, 2012.
	\newblock An introduction to geometric measure theory.
	
	\bibitem[Men13]{Menne13}
	Ulrich Menne.
	\newblock Second order rectifiability of integral varifolds of locally bounded
	first variation.
	\newblock {\em J. Geom. Anal.}, 23(2):709--763, 2013.
	
	\bibitem[MP20]{MagnaniniPoggesi}
	Rolando Magnanini and Giorgio Poggesi.
	\newblock Nearly optimal stability for {S}errin's problem and the soap bubble
	theorem.
	\newblock {\em Calc. Var. Partial Differential Equations}, 59(1):Paper No. 35,
	23, 2020.
	
	\bibitem[MR91]{MontielRos}
	Sebasti\'{a}n Montiel and Antonio Ros.
	\newblock Compact hypersurfaces: the {A}lexandrov theorem for higher order mean
	curvatures.
	\newblock In {\em Differential geometry}, volume~52 of {\em Pitman Monogr.
		Surveys Pure Appl. Math.}, pages 279--296. Longman Sci. Tech., Harlow, 1991.
	
	\bibitem[MS23]{MaggiSantilli}
	Francesco Maggi and Mario Santilli.
	\newblock Rigidity and compactness with constant mean curvature in warped
	product manifolds, 2023.
	
	\bibitem[Res68]{Reshetnyakdifferentiability}
	Ju.~G. Reshetnyak.
	\newblock Generalized derivatives and differentiability almost everywhere.
	\newblock {\em Mat. Sb. (N.S.)}, pages 323--334, 1968.
	
	\bibitem[Ros87]{RosRevista}
	Antonio Ros.
	\newblock Compact hypersurfaces with constant higher order mean curvatures.
	\newblock {\em Rev. Mat. Iberoamericana}, 3(3-4):447--453, 1987.
	
	\bibitem[Ros88]{KorevaarRos}
	Antonio Ros.
	\newblock Compact hypersurfaces with constant scalar curvature and a congruence
	theorem.
	\newblock {\em J. Differential Geom.}, 27(2):215--223, 1988.
	\newblock With an appendix by Nicholas J. Korevaar.
	
	\bibitem[RZ19]{RatajZaehlebook}
	Jan Rataj and Martina Z\"{a}hle.
	\newblock {\em Curvature measures of singular sets}.
	\newblock Springer Monographs in Mathematics. Springer, Cham, 2019.
	
	\bibitem[San20a]{SantilliAnnali}
	Mario Santilli.
	\newblock Fine properties of the curvature of arbitrary closed sets.
	\newblock {\em Ann. Mat. Pura Appl. (4)}, 199(4):1431--1456, 2020.
	
	\bibitem[San20b]{SantilliBulletin}
	Mario Santilli.
	\newblock Normal bundle and {A}lmgren's geometric inequality for singular
	varieties of bounded mean curvature.
	\newblock {\em Bull. Math. Sci.}, 10(1):2050008, 24, 2020.
	
	\bibitem[Sim83]{Simonbook}
	Leon Simon.
	\newblock {\em Lectures on geometric measure theory}, volume~3 of {\em
		Proceedings of the Centre for Mathematical Analysis, Australian National
		University}.
	\newblock Australian National University, Centre for Mathematical Analysis,
	Canberra, 1983.
	
\end{thebibliography}
\end{document}